\documentclass[reqno,oneside]{amsart}

\usepackage{mathtools}
\usepackage{amssymb}
\usepackage[margin=1in]{geometry}
\usepackage{graphicx}
\usepackage{subcaption}
\usepackage[shortlabels]{enumitem}
\usepackage{hyperref}
\usepackage{tikz}
\usetikzlibrary{calc}

\graphicspath{{./images/}}

\renewcommand{\phi}{\varphi}
\renewcommand{\epsilon}{\varepsilon}
\newcommand{\prm}{\ell}
\newcommand{\T}{\mathbb{T}}
\newcommand{\R}{\mathbb{R}}
\newcommand{\C}{\mathbb{C}}
\newcommand{\Z}{\mathbb{Z}}

\DeclareMathOperator{\im}{im}
\DeclareMathOperator{\lcm}{lcm}
\DeclareMathOperator{\ord}{ord}
\DeclareMathOperator{\SU}{SU}
\DeclareMathOperator{\Tr}{Tr}

\newtheorem{thm}{Theorem}[section]
\newtheorem{lem}[thm]{Lemma}
\newtheorem{prop}[thm]{Proposition}
\newtheorem*{cor}{Corollary}

\title[Graphical cyclic supercharacters for composite moduli]
	{Graphical cyclic supercharacters for composite moduli}
\author[B. Lutz]{Bob Lutz}
\address{Department of Mathematics,
	University of Michigan,
	2074 East Hall,
	530 Church Street,
	Ann Arbor, MI 48109,
	USA}
\email{boblutz@umich.edu}
\urladdr{\url{http://www-personal.umich.edu/~boblutz}}

\begin{document}
\maketitle

\begin{abstract}
Recent work has introduced the study of graphical properties of cyclic supercharacters, functions $\Z/n\Z\to \C$ whose values are exponential sums with close connections to Gauss sums and Gaussian periods. Plots of these functions exhibit striking features, some of which have been previously explained when the modulus $n$ is a power of an odd prime. After reviewing this material, we initiate the graphical study of images of cyclic supercharacters in the case of composite $n$.
\end{abstract}

\section{Introduction}

For a positive integer $n$ and a unit $\omega$ mod $n$ of order $d$, the associated \emph{cyclic supercharacter mod $n$} is the function $\sigma_\omega:\Z/n\Z\to \C$ given by
\begin{equation*}
	\sigma_\omega(y) = \sum_{j=1}^d e\left(\frac{\omega^j y}{n}\right),
\end{equation*}
where $e(\theta):=\exp(2\pi i \theta)$ for all real $\theta$. Gauss studied the values of cyclic supercharacters mod primes $p>2$, called \emph{Gaussian periods}, as they relate to the problem of drawing regular polygons with compass and straight-edge. These values are modernly called \emph{Gaussian periods} and have appeared in many contexts, including the construction of difference sets and the optimized AKS algorithm of Lenstra and Pomerance \cite{baumert1971,lenstra2002}. A more detailed account of the history of Gaussian periods with references can be found in \cite{hyde2015}, although our notation differs from theirs.

\begin{figure}[ht]
	\begin{subfigure}[b]{0.3\textwidth}
		\includegraphics[width=\textwidth]{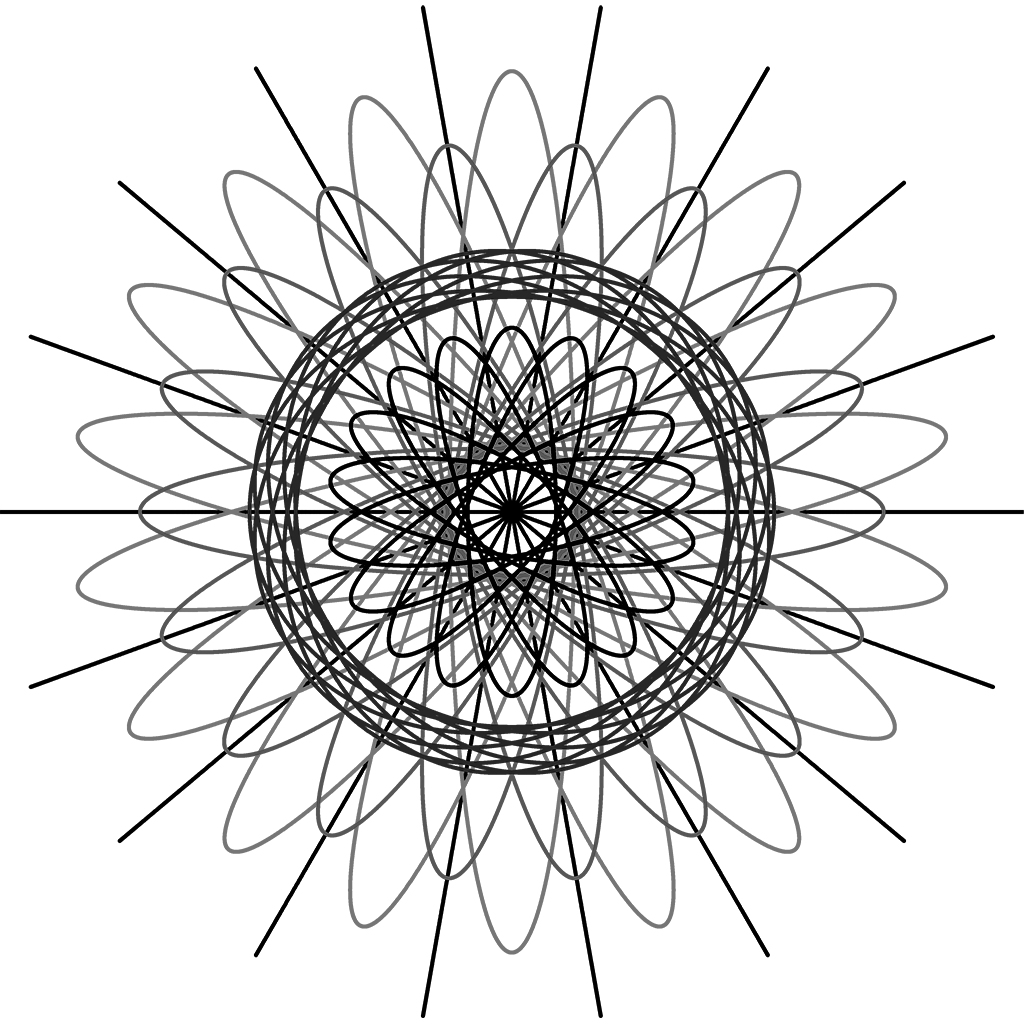}
		\caption{\scriptsize $n=478125$, $\omega=3124$}
	\end{subfigure}
	\quad
	\begin{subfigure}[b]{0.3\textwidth}
		\includegraphics[width=\textwidth]{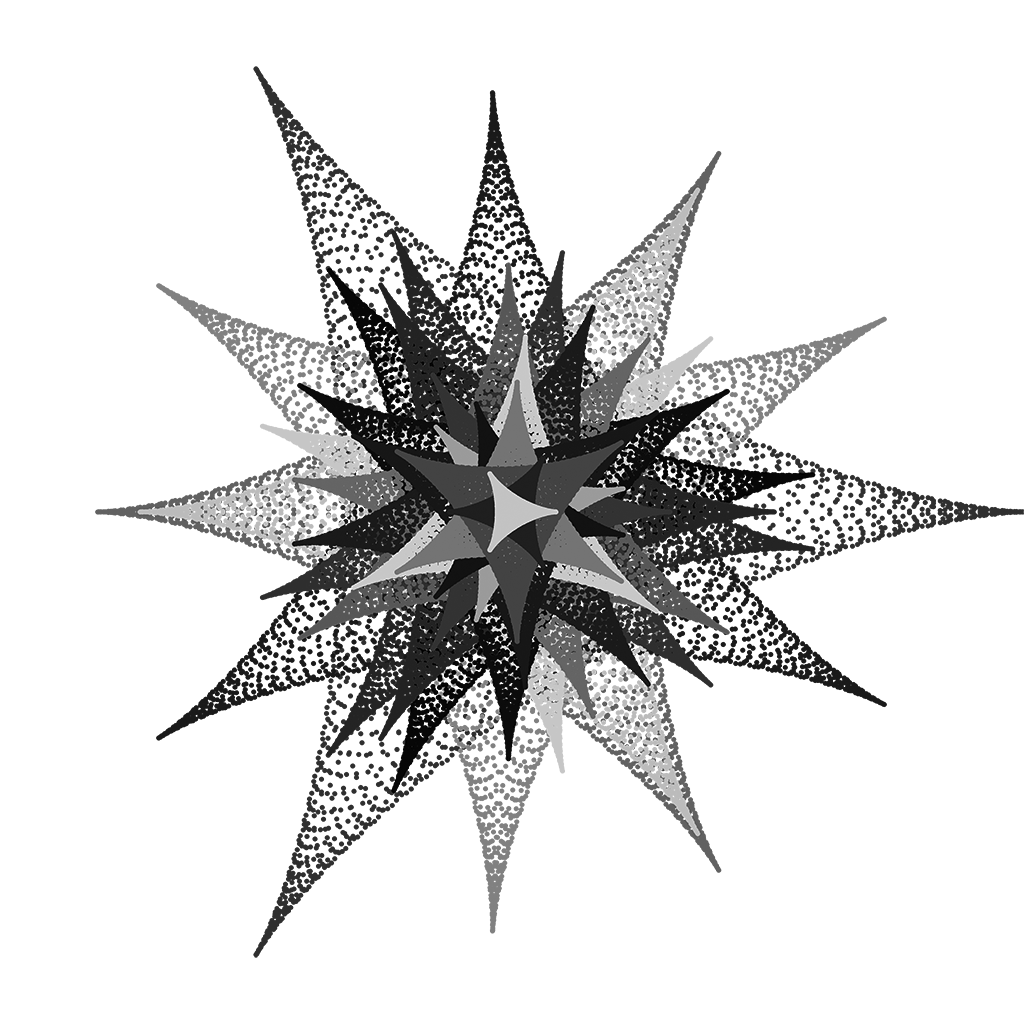}
		\caption{\scriptsize $n=551905$, $\omega=20719$}
	\end{subfigure}
	\quad
	\begin{subfigure}[b]{0.3\textwidth}
		\includegraphics[width=\textwidth]{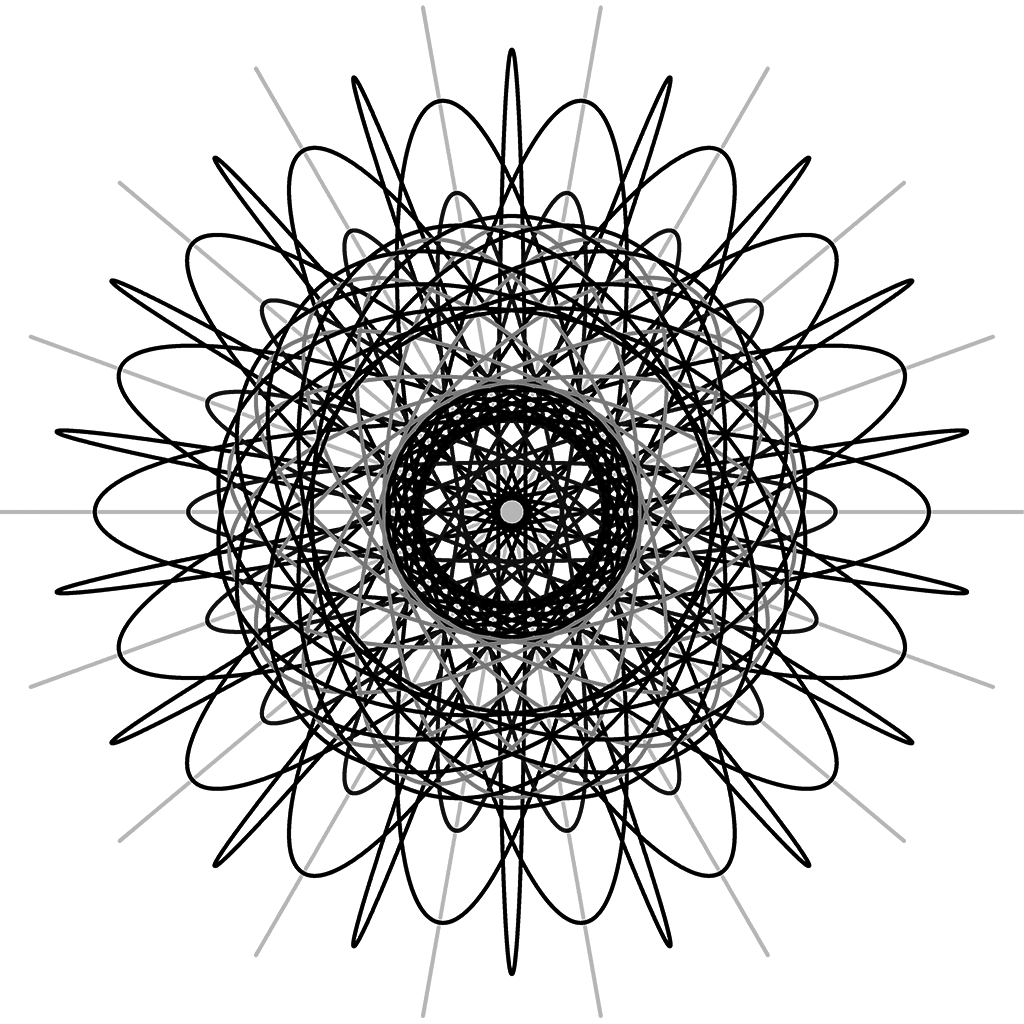}
		\caption{\scriptsize $n=455175$, $\omega=107218$}
	\end{subfigure}
	\caption{Realized as complex plots, the images of cyclic supercharacters $\sigma_\omega$ mod $n$ reveal themselves in surprising ways.}
	\label{prettyfig}
\end{figure}

Kummer introduced analogous sums, values of cyclic supercharacters $\sigma_\omega$ mod $n$, for composite $n$. These sums have been studied in their own right and linked to certain difference sets \cite{evans1981,evans1982,lehmer1981}. While individual values can be difficult to analyze, recent work has revealed striking and accessible patterns in these sums when viewed together as the image $\im(\sigma_\omega)$ of a cyclic supercharacter for a fixed modulus $n$ and generator $\omega$. Figure \ref{prettyfig} offers a small gallery of such images as complex plots.

For $n$ a power of an odd prime, much of this graphical behavior has been described previously in terms of certain Laurent polynomials on high-dimensional tori \cite{duke2015,hyde2015}. We review this material briefly in Section \ref{ppmodsec}. Comparatively little, however, has been done to study the analogous properties of cyclic supercharacters mod non-prime-power $n$. With this note, we aim to explain concisely and systematically many of the patterns yet observed in the images of these supercharacters.

For convenience, we will frequently consider cyclic supercharacters $\sigma_\omega$ as periodic functions on $\Z$ with period $n$, and treat integers tacitly as residues whenever it does not affect the statement. The functions $\sigma_\omega$ are supercharacters in the sense of \cite{diaconis2008}, but we do not adopt this perspective here. Ramanujan sums, Heilbronn sums, and generalized Kloosterman sums can also be viewed as values of supercharacters \cite{brumbaugh2014,burkhardt2015,garcia2015}. For cyclic supercharacters, the motivated reader can find the details of this approach in \cite{duke2015}.

\section{Prime-power moduli}
\label{ppmodsec}

In this section, we consider cyclic supercharacters mod $p^a$ for an odd prime $p$ and positive integer $a$. There is a description of the images of such supercharacters in terms of the images of certain Laurent polynomials, which we record below. Throughout, we write $\phi$ for the totient function, $\T$ for the unit circle in $\C$, and $\Phi_d(x)$ for the $d$th cyclotomic polynomial in $x$. Recall that $\Phi_d(x)$ is monic and has all integer coefficients. The following result is due to \cite{duke2015}.

\begin{thm}
	\label{dukethm}
	Fix a positive integer $d$. If $p\equiv 1\pmod d$ is an odd prime and $\omega$ is a unit of order $d$ mod $p^a$ for some positive integer $a$, then $\im \sigma_\omega$ is contained in the image of the function $g_d:\T^{\phi(d)}\to \C$ given by
	\begin{equation*}
		g_d(z_1,\ldots,z_{\phi(d)})=\sum_{k=1}^{d}\prod_{j=1}^{\phi(d)} z_j^{c_{j,k}},
	\end{equation*}
	where the exponents $c_{j,k}$ are integers determined by the relations
	\begin{equation*}
		x^k \equiv \sum_{j=1}^{\phi(d)} c_{j,k}x^{j-1} \pmod{\Phi_d(x)}.
	\end{equation*}
	Moreover, every open disk in the image of $g_d$ contains points in the images of $\sigma_\omega$ for sufficiently large $p^a$ subject to $p\equiv 1\pmod d$.
\end{thm}

For $k=1,2,\ldots$ let $A_k\subset \C$. If there exists a set $B\subset \C$ such that $A_k\subset B$ for all $k$ and, for each nonempty open set $U\subset B$, a positive integer $k_U$ for which $k>k_U$ implies that $U\cap A_{k_U}$ is nonempty, then we say that the sets $A_k$ \emph{fill out} $B$ as $k\to\infty$. In these terms, we can rephrase the last statement of Theorem \ref{dukethm} by saying that the images $\im(\sigma_\omega)$ fill out $\im g_d$ as $p^a\to\infty$ subject to $p\equiv 1\pmod{d}$.

\begin{figure}[ht]
	\begin{subfigure}[b]{0.21\textwidth}
		\includegraphics[width=\textwidth]{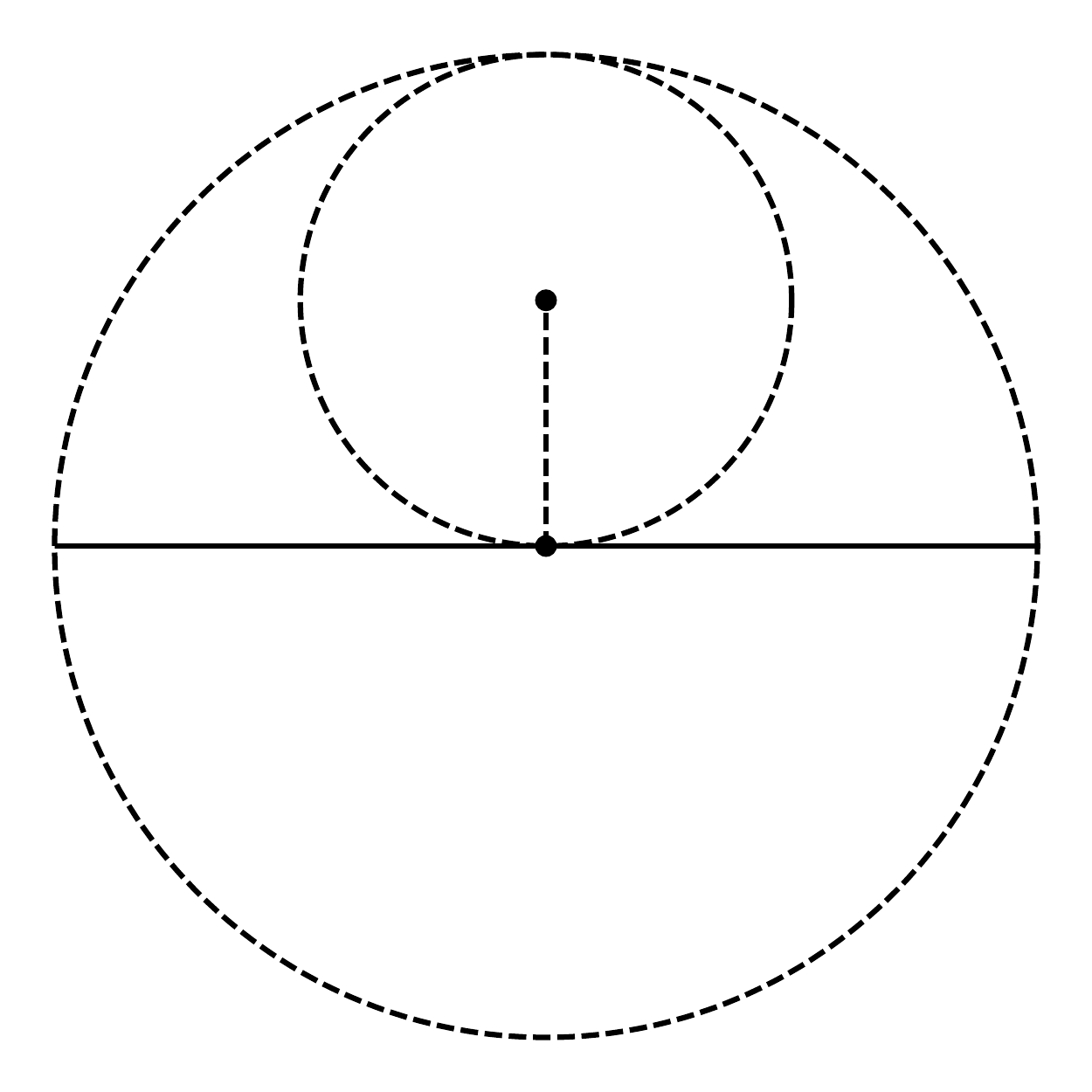}
		\caption{\scriptsize Tusi couple}
	\end{subfigure}
	\quad
	\begin{subfigure}[b]{0.21\textwidth}
		\includegraphics[width=\textwidth]{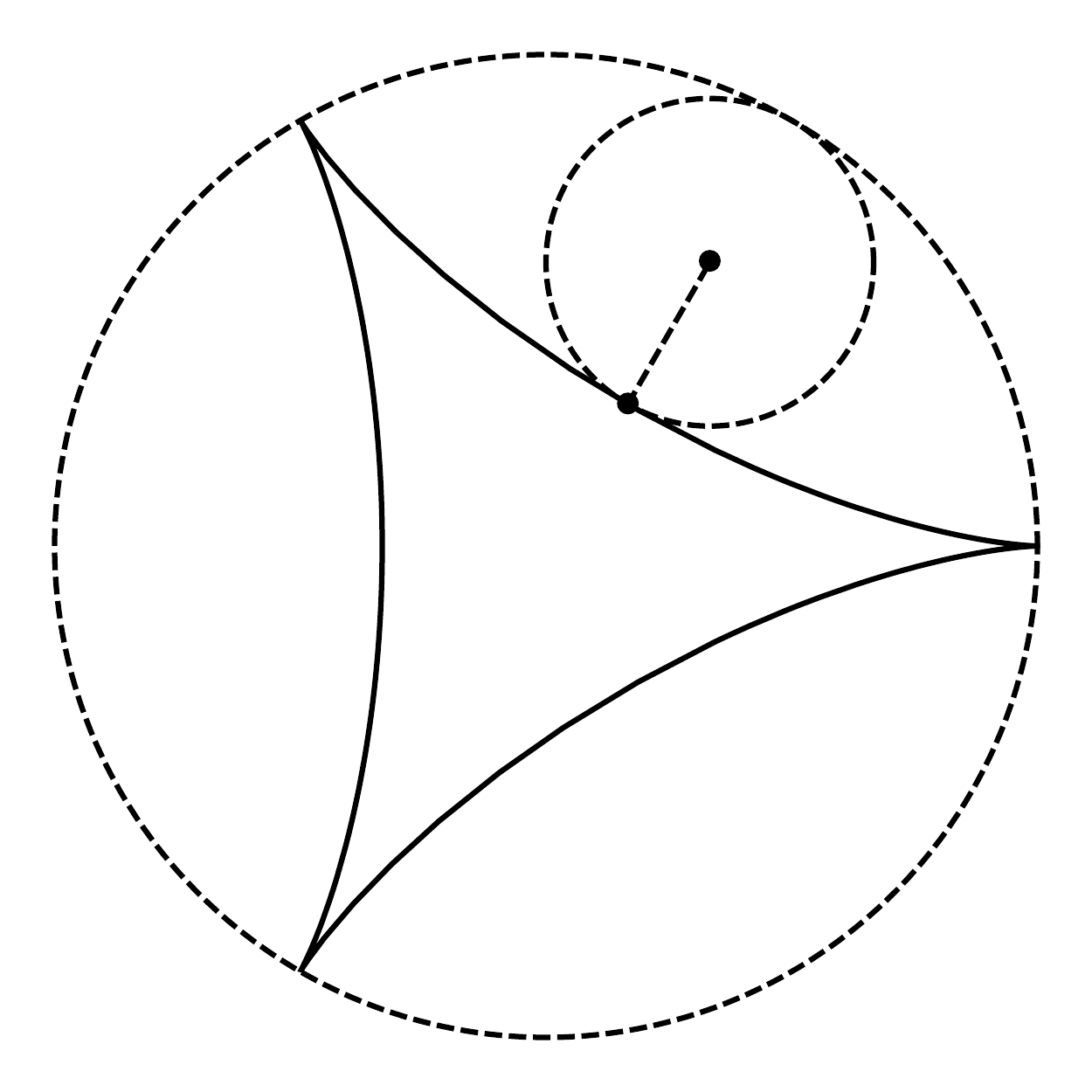}
		\caption{\scriptsize Deltoid}
	\end{subfigure}
	\quad
	\begin{subfigure}[b]{0.21\textwidth}
		\includegraphics[width=\textwidth]{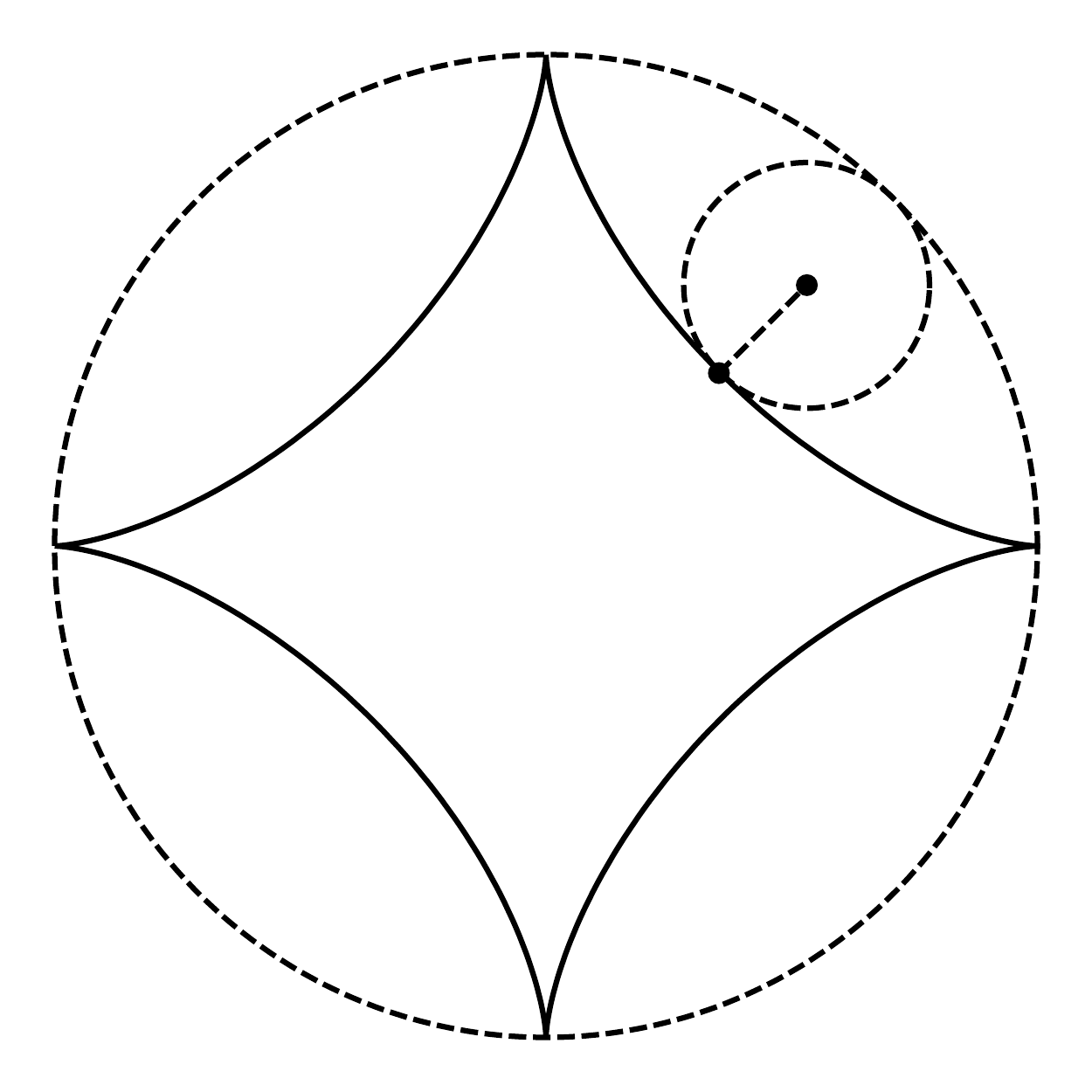}
		\caption{\scriptsize Astroid}
	\end{subfigure}
	\quad
	\begin{subfigure}[b]{0.21\textwidth}
		\includegraphics[width=\textwidth]{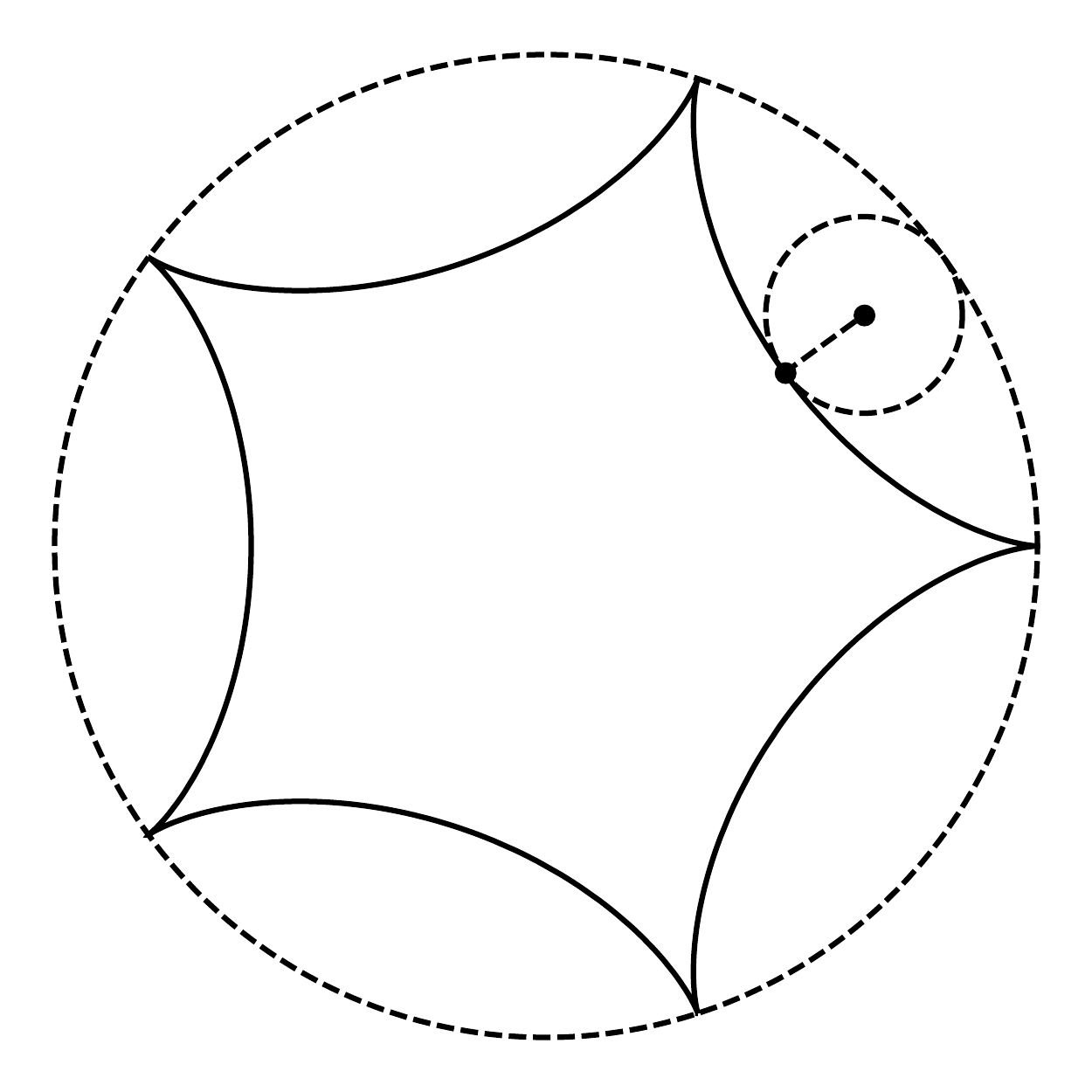}
		\caption{\scriptsize 5-hypocycloid}
	\end{subfigure}
	\caption{A circle of radius 1 traces out hypocycloids as it rolls within circles of radii $2$, $3$, $4$, and $5$.}
	\label{rollfig}
\end{figure}

The clearest behavior occurs when, in the notation of Theorem \ref{dukethm}, $d$ is a positive power of an odd prime. Recall that a \emph{hypocycloid} is a planar curve obtained by tracing a fixed point on a circle as it rolls within a larger circle. This construction, illustrated in Figure \ref{rollfig}, produces a simple closed curve if the smaller radius divides the larger; the number of cusps is the ratio of the larger radius to the smaller. For all integers $k\geq 2$, let $H_k\subset\C$ denote the compact, simply-connected set whose boundary is the $k$-cusped hypocycloid centered at 0 with a cusp at $k$. Let $P_k$ denote the convex hull of $H_k$, whose boundary is the regular $k$-gon centered at 0 with a vertex at $k$.

\begin{figure}[ht]
	\begin{subfigure}[b]{0.3\textwidth}
		\includegraphics[width=\textwidth]{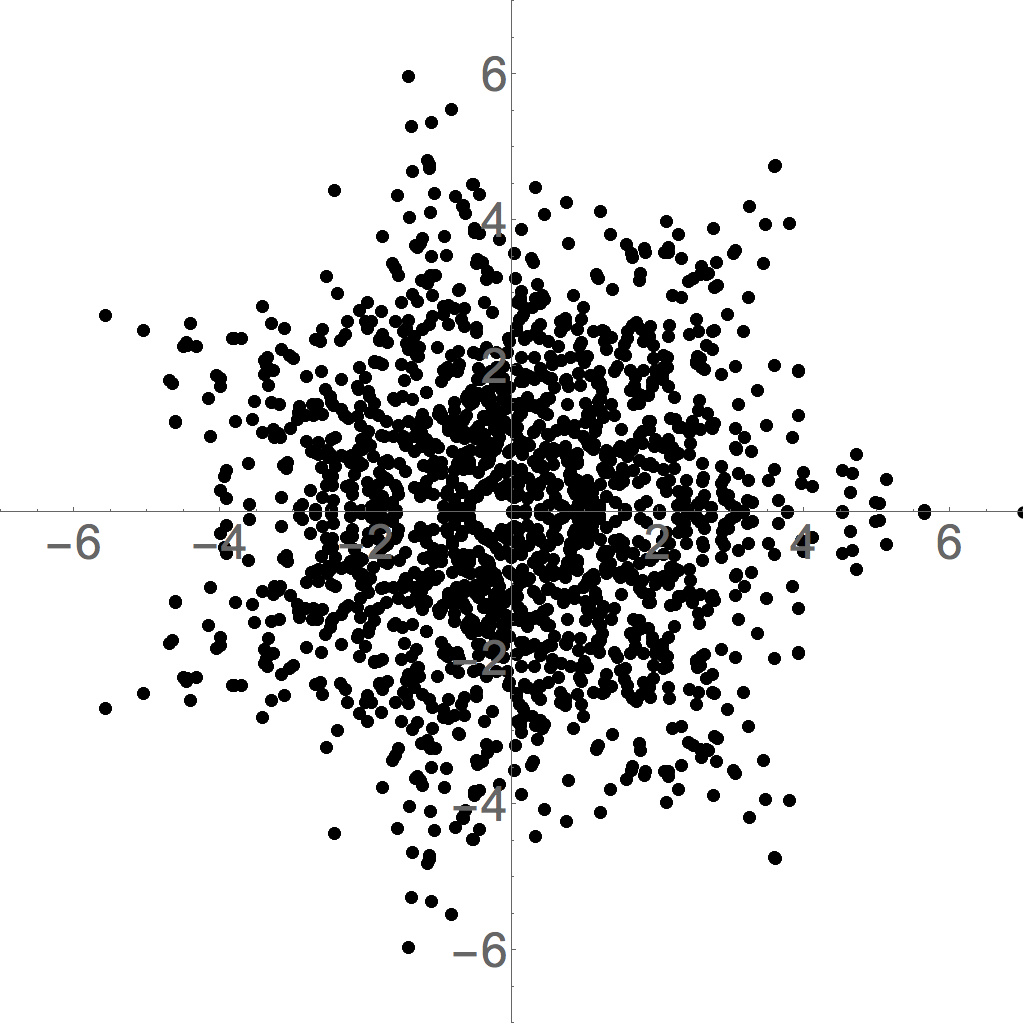}
		\caption{\scriptsize $(113^2,129)$}
	\end{subfigure}
	\quad
	\begin{subfigure}[b]{0.3\textwidth}
		\includegraphics[width=\textwidth]{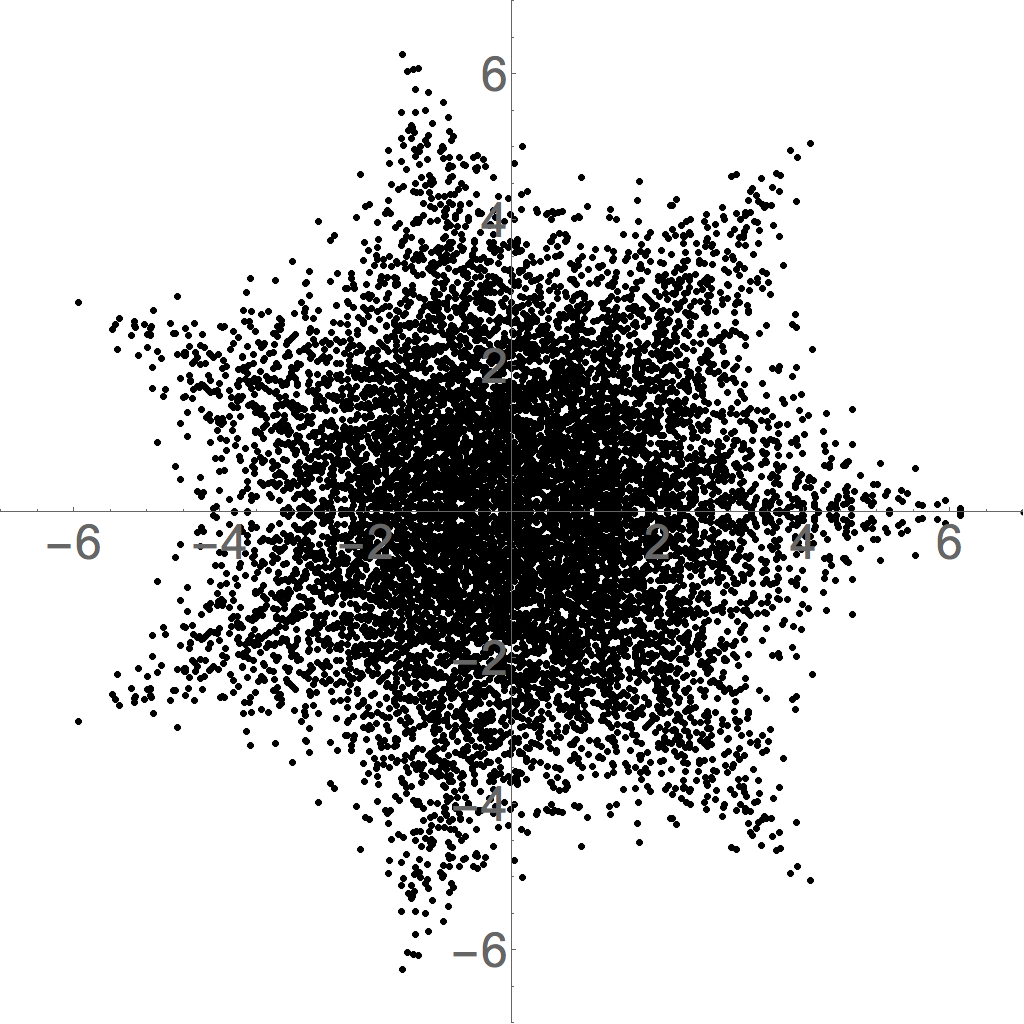}
		\caption{\scriptsize $(43^3,3623)$}
	\end{subfigure}
	\quad
	\begin{subfigure}[b]{0.3\textwidth}
		\includegraphics[width=\textwidth]{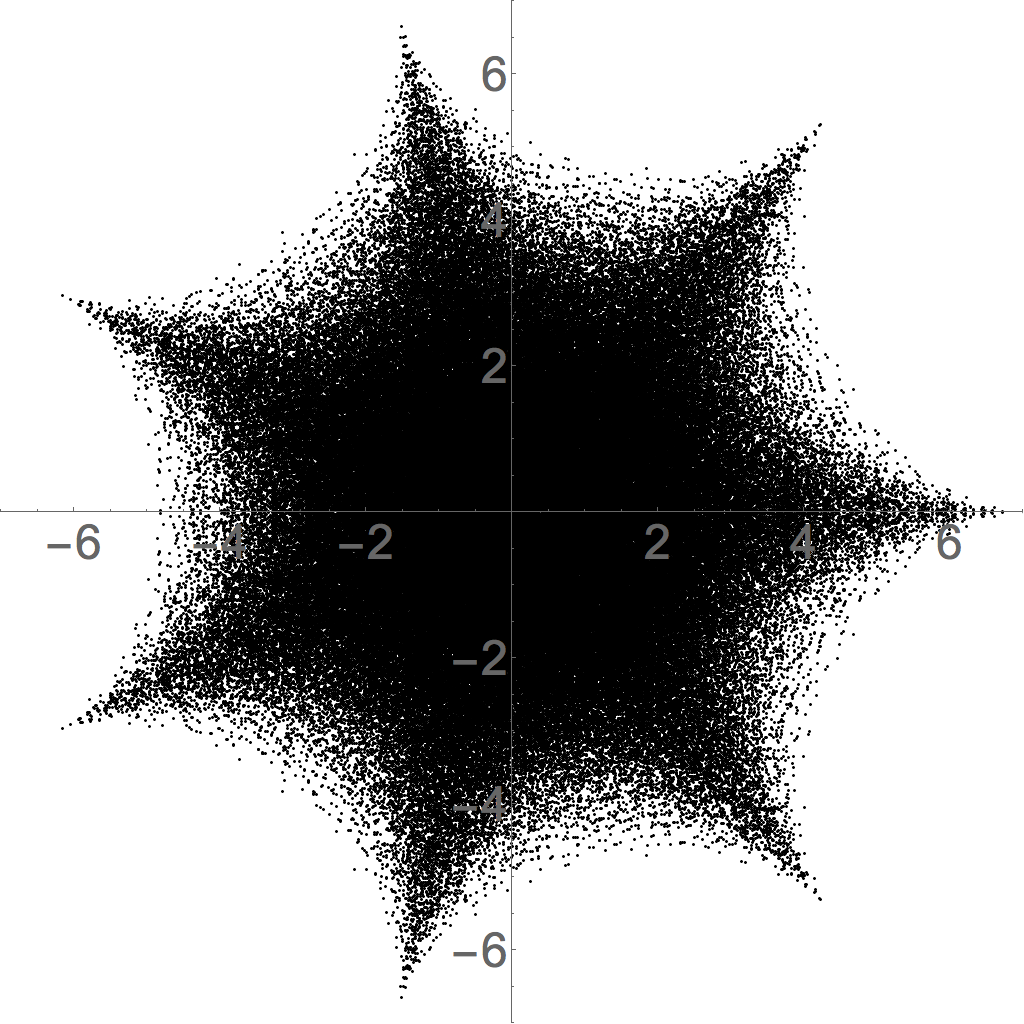}
		\caption{\scriptsize $(1289^2,341010)$}
	\end{subfigure}
	\caption{For the given pairs $(n,\omega)$, the images of the cyclic supercharacters $\sigma_\omega$ mod $n$ are on their way to filling out $H_7$.}
	\label{hypofig}
\end{figure}

If $\prm$ is an odd prime, then $\phi(\prm)=\prm-1$ and $\Phi_\prm(x)=1+x+x^2+\cdots+x^{\prm-1}$, so
\begin{equation*}g_\prm(z_1,\ldots,z_{\prm-1})=z_1+z_2+\cdots+z_{\prm-1}+\frac{1}{z_1z_2\cdots z_{\prm-1}}.\end{equation*}
The image of $g_\prm$ is seen to be $\Tr(\SU_\prm(\C))$, which is precisely $H_\prm$ \cite[Theorem 3.2.3]{cooper2007}. More is true, but we require additional notation to write it succinctly. In Figure \ref{hypofig}, several terms of a sequence filling out $H_7$ are illustrated.

For nonempty subsets $A$ and $B$ of $\C$, make the definitions
\begin{equation}
	\begin{aligned}
		A\oplus B &= \{a+b : (a,b)\in A\times B\}\\
		A\otimes B &=\{ab : (a,b)\in A\times B\}.
	\end{aligned}
	\label{eqmink}
\end{equation}
The sets in \eqref{eqmink} are sometimes called the \emph{Minkowski sum} and \emph{Minkowski product}, respectively, of $A$ and $B$, and the operations $\oplus$ and $\otimes$ are called \emph{Minkowski addition} and \emph{Minkowski multiplication}. The corresponding $n$-ary operations are defined by induction; for convenience, we write $A\oplus \cdots \oplus A$ as $A^{\oplus k}$, where $k$ is the number of summands. While Minkowski addition and multiplication are both commutative and have identity elements, neither distributes over the other or has a well-defined inverse operation. Minkowski addition has been studied extensively in Euclidean space and is well understood, at least compared to Minkowski multiplication, which is an active subject of research in pure and applied settings \cite{farouki2000,farouki2001,li2015}.

\begin{figure}[ht]
	\begin{subfigure}[b]{0.3\textwidth}
		\includegraphics[width=\textwidth]{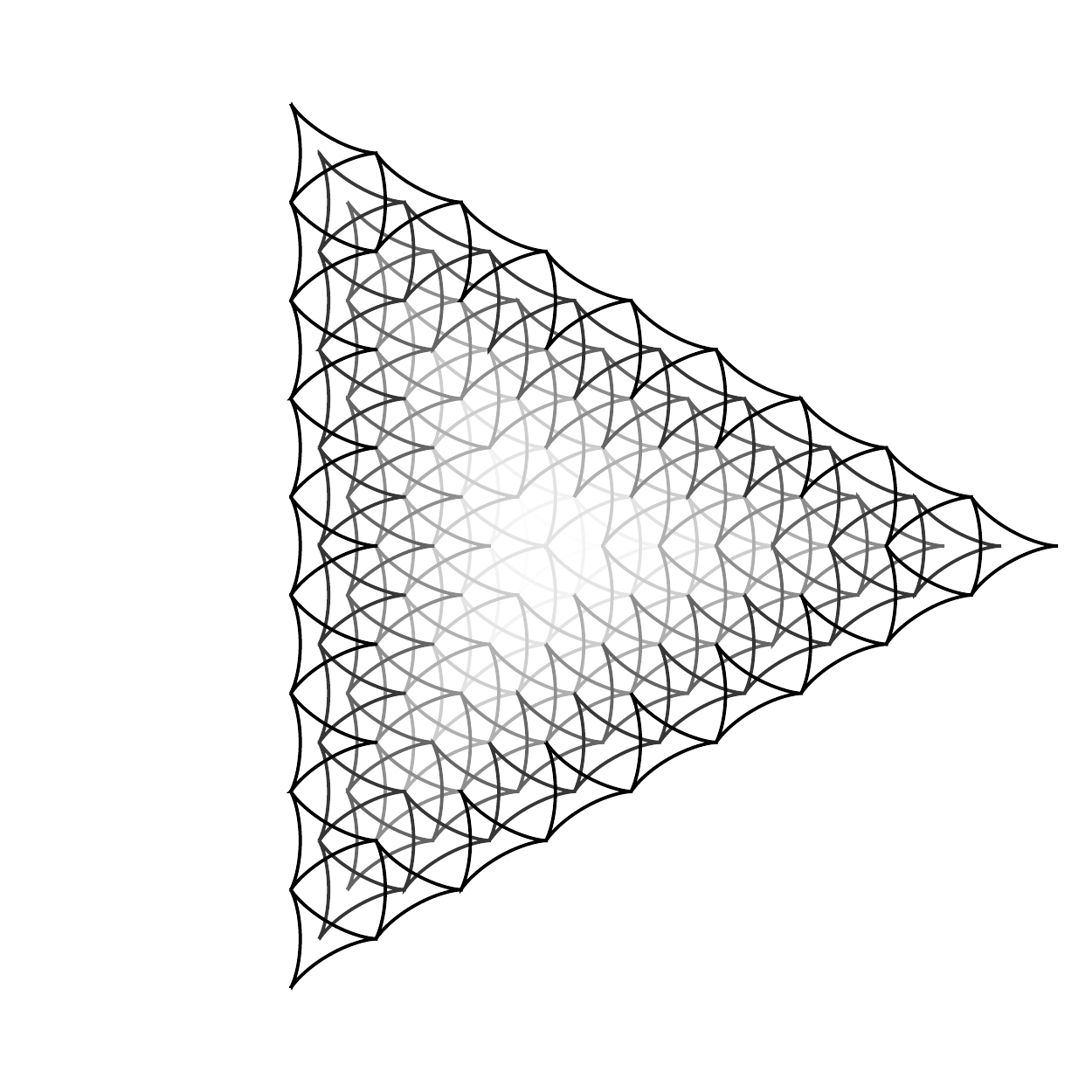}
		\caption{\scriptsize $\prm=b=3$}
	\end{subfigure}
	\quad
	\begin{subfigure}[b]{0.3\textwidth}
		\includegraphics[width=\textwidth]{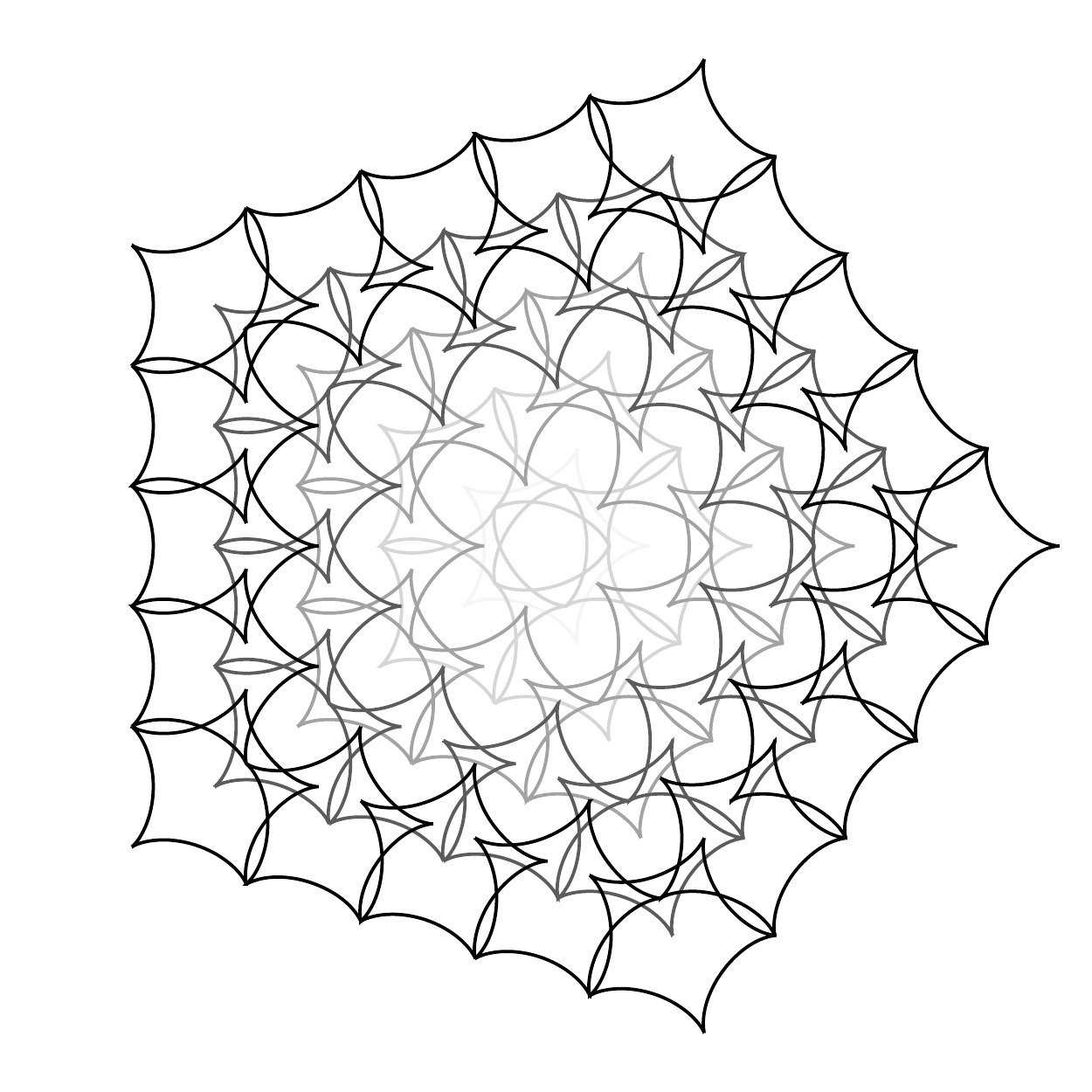}
		\caption{\scriptsize $\prm=5$, $b=2$}
	\end{subfigure}
	\quad
	\begin{subfigure}[b]{0.3\textwidth}
		\includegraphics[width=\textwidth]{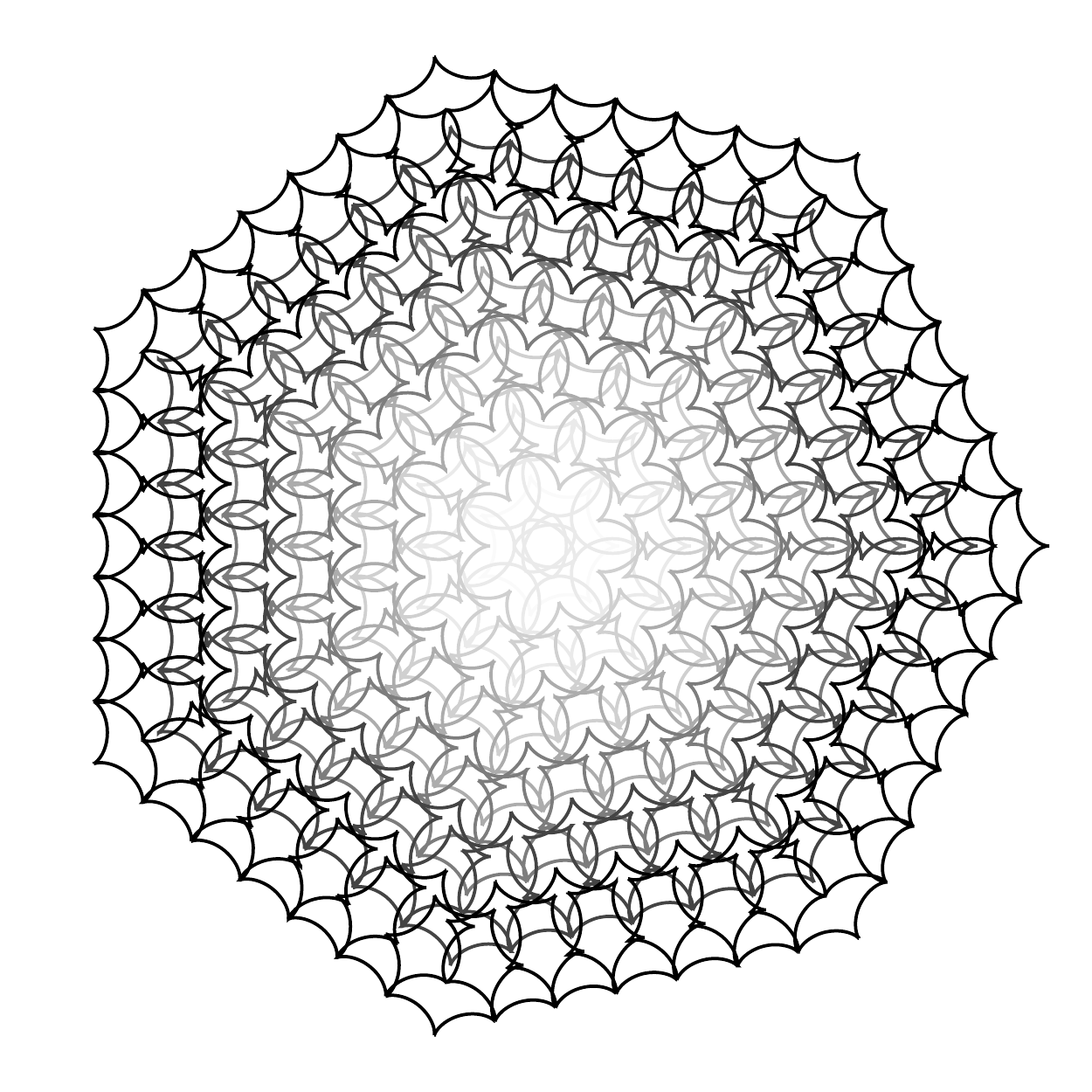}
		\caption{\scriptsize $\prm=7$, $b=2$}
	\end{subfigure}
	\caption{The outer boundaries in the figures form the boundaries of $H_\prm^{\oplus \prm^{b-1}}$.}
	\label{minksumfig}
\end{figure}

For the moment, we are concerned with Minkowski addition. If $\prm^b$ is a positive power of an odd prime $\prm$, then it can be shown that
\begin{equation}
	\im(g_{\prm^b})=H_\prm^{\oplus \prm^{b-1}}.
	\label{sumseteq}
\end{equation}
Several of these sets are illustrated in Figure \ref{minksumfig}. The reader might notice that as $\prm^b$ increases, the figures begin to resemble regular polygons. Indeed, it follows from a corollary to the Shapley--Folkman theorem in \cite{starr1969} that as $k\to\infty$, the scaled Minkowski sums $\frac{1}{k}H_\prm^{\oplus k}$ fill out $P_\prm$. To close the section, we record the corresponding implication for cyclic supercharacters. The proof is an application of the preceding discussion to Theorem \ref{dukethm}.

\begin{prop}
	Fix an odd prime $\prm$. For $k=1,2,\ldots$ let $b_k$ be a positive integer, $p_k>\prm$ an odd prime with $\prm^{b_k} | \phi(p_k^{a_k})$, and $\omega_k$ a unit mod $p_k^{a_k}$ of order $\prm^{b_k}$. As $k\to\infty$, if $b_k\to\infty$, then the scaled images $\prm^{1-b_k}\im( \sigma_{\omega_k})$ fill out $P_\prm$.
	\label{polyprop}
\end{prop}

\section{Composite moduli}

We turn our attention now to cyclic supercharacters whose moduli are not a power of a prime. Let $a$ and $b$ be integers. For a unit $\omega$ mod $a$, we denote by $\ord(\omega)$ the (multiplicative) order of $\omega$. Unless indicated otherwise, $(a,b)$ will denote the GCD of $a$ and $b$. If $b|a$, then unless necessary, we will not distinguish between $\omega$ and its image under the reduction map $\Z/a\Z\to \Z/b\Z$. When we wish to emphasize the change in modulus, we shall write the residue of $\omega$ mod $b$ as $\omega_b$.

\subsection{General behavior}

We recall some elementary geometric notions. A set $A\subset \C$ is said to have \emph{$k$-fold dihedral symmetry} if it is invariant under the action on $\C$ of the dihedral group of order $2k$ by complex conjugation and rotation by $2\pi/k$ about the origin. The intersection of all supersets of $A$ having $k$-fold dihedral symmetry is called the \emph{$k$-fold dihedral closure} of $A$. Equivalently, this is the union of the orbits of all points in $A$. If $A$ is closed under complex conjugation, then its $k$-fold dihedral closure is
\begin{equation}
	\{e(j/k): j=1,\ldots,k\}\otimes A.
	\label{diheq}
\end{equation}
Proposition \ref{symprop1} below is an extension of \cite[Proposition 3.1]{duke2015}. Proposition \ref{symprop2} is a useful observation in the vein of Section \ref{ppmodsec}.

\begin{prop}
	Let $\sigma_\omega$ be a cyclic supercharacter mod $n$, and write $k=(\omega-1,n)$.
	\begin{enumerate}[label={(\alph*)},ref={\thethm(\alph*)}]
		\item The $k$-fold dihedral closure of $\im(\sigma_{\omega_{n/k}})$ is $\im(\sigma_\omega)$.
		\label{symprop1}
		\item If $k=1$ and $\ord(\omega)>1$, then $\im(\sigma_\omega)\subset H_{\ord(\omega)}$.
		\label{symprop2}
	\end{enumerate}

	\begin{proof}
		Since $k=(\omega-1,n)$, we have $\ord(\omega_{n/k})=\ord(\omega)$. For $j=1,\ldots,\ord(\omega)$, write $\omega^j=1+r_jk$ and notice that
		\begin{equation*}
			\sigma_\omega\left(y+n/k\right)=\sum_{j=1}^{\ord(\omega)} e\left(\frac{(1+r_j k)(y+n/k)}{n}\right)=e(1/k)\sigma_{\omega_{n/k}}(y),
		\end{equation*}
		since $\ord(\omega)=\ord(\omega_{n/k})$. Combine this with the fact that $\sigma_\omega(-y)=\overline{\sigma_\omega(y)}$ to obtain (a). For (b), notice that $\omega+\omega^2+\cdots+\omega^{\ord(\omega)}=0$, so 
		\begin{equation*}
			\sigma_\omega(y) = e\left(\frac{-(\omega+\cdots+\omega^{\ord(\omega)-1})y}{n}\right)
			+ \sum_{j=1}^{\ord(\omega)-1} e\left(\frac{\omega^jy}{n}\right).
		\end{equation*}
		In particular, $\im(\sigma_\omega)\subset\Tr(\SU_{\ord(\omega)}(\C))$. Appealing to \cite[Theorem 3.2.3]{cooper2007} completes the proof.
	\end{proof}
\end{prop}

\subsection{A new perspective}

Many patterns recognizable in the plots of cyclic supercharacters can be explained by the following overlooked mechanism. The remainder of the article is dedicated to consequences of Theorem \ref{splitthm}.

\begin{thm}
	Suppose that $\sigma_\omega$ is a cyclic supercharacter mod $mn$ for positive integers $m$ and $n$. If $\ord(\omega_n)=uv$ where $(v,\ord(\omega_m))=1$, then
	\begin{equation*}
		\sigma_\omega(sm+tn)=\sum_{j=1}^{u} \sigma_{\omega_m^u}(\omega^jt)\sigma_{\omega_n^u}(\omega^j s),
	\end{equation*}
	for all integers of the form $sm+tn$.

	\begin{proof}
		Let $d$ be the order of $\omega$. We have
		\begin{align*}
			\sigma_{\omega}(sm+tn)
			&=\sum_{j=1}^d e\left(\frac{\omega^j (sm+tn)}{mn}\right)\\
			&=\sum_{j=1}^{uv}\sum_{k=1}^{d/(uv)} e\left(\frac{\omega^{j+uvk}s}{n}\right)e\left(\frac{\omega^{j+uvk}t}{m}\right)\\
			&=\sum_{j=1}^{uv} e\left(\frac{\omega^js}{n}\right) \sum_{k=1}^{d/(uv)} e\left(\frac{\omega^j\omega^{uvk}t}{m}\right)\\
			&=\sum_{j=1}^{uv} e\left(\frac{\omega^js}{n}\right) \sigma_{\omega_m^{uv}}(\omega^jt).
		\end{align*}
		Since $(v,\ord(\omega_m))=1$, we have $\omega_m^{uv}=\omega_m^u$. Moreover, it is not difficult to show that $\sigma_{\omega_m^u}(\omega^jt)$ depends only on the residue of of $j$ mod $u$. Hence
		\begin{equation*}
			\sigma_{\omega}(sm+tn)
			=\sum_{j=1}^u \sigma_{\omega_m^u}(\omega^jt)\sum_{k=1}^{v} e\left(\frac{\omega^{j+ku}s}{n}\right)
			=\sum_{j=1}^u \sigma_{\omega_m^u}(\omega^jt)\sigma_{\omega_n^u}(\omega^j s).
			\qedhere
		\end{equation*}
	\end{proof}
	\label{splitthm}
\end{thm}

\begin{cor}
	If $u=1$ in the above notation, so $(\ord(\omega_m),\ord(\omega_n))=1$, then
	\begin{equation*}
		\sigma_\omega(sm+tn)=\sigma_{\omega_m}(t)\sigma_{\omega_n}(s).
	\end{equation*}
	In particular, $\im(\sigma_\omega)\supset\im(\sigma_{\omega_m})\otimes \im(\sigma_{\omega_n})$ with equality whenever $(m,n)=1$.
\end{cor}

Induction on the corollary yields \cite[Theorem 2.1]{duke2015}, although the statement there lacks a necessary hypothesis. Applying this fact to the discussion in Section \ref{ppmodsec} gives the following result, which connects images of cyclic supercharacters with Minkowski products of hypocycloids and regular polygons.

\begin{prop}
	Fix a positive integer $k$ and distinct odd primes $\prm_1,\ldots,\prm_k$. For each $j=1,\ldots,k$, let $A_j$ be either $P_{\prm_j}$ or $H_{\prm_j}^{\oplus b_j}$ for some positive integer $b_j$. There is a sequence of cyclic supercharacters whose images, when scaled appropriately, fill out $A_1\otimes\cdots\otimes A_k$. Moreover, scaling is only necessary if $A_j=P_{\prm_j}$ for some $j$.
	\label{prodsumprop}
\end{prop}

In Figure \ref{mpfig} we plot individual terms of sequences described in Proposition \ref{prodsumprop}, where $k=2$ and $A_1=H_3$. While the boundary of the Minkowski product $H_{\prm_1}\otimes H_{\prm_2}$ ought to have $\prm_1\prm_2$ cusps, each of the plots in Figures \ref{mpfig2} and \ref{mpfig3} exhibits only $3$. This is because the values of $\sigma_{\omega_n}$ are concentrated toward the origin and hence far from the non-real cusps of $H_{\prm_2}$. In order for the image of $\sigma_\omega$ to resemble $H_3\otimes H_{\prm_2}$ visually, larger values of $n$ are necessary. In Figure \ref{mpfig1}, the expected $15$ cusps are more evident.

\begin{figure}[ht]
	\begin{subfigure}[b]{0.3\textwidth}
		\includegraphics[width=\textwidth]{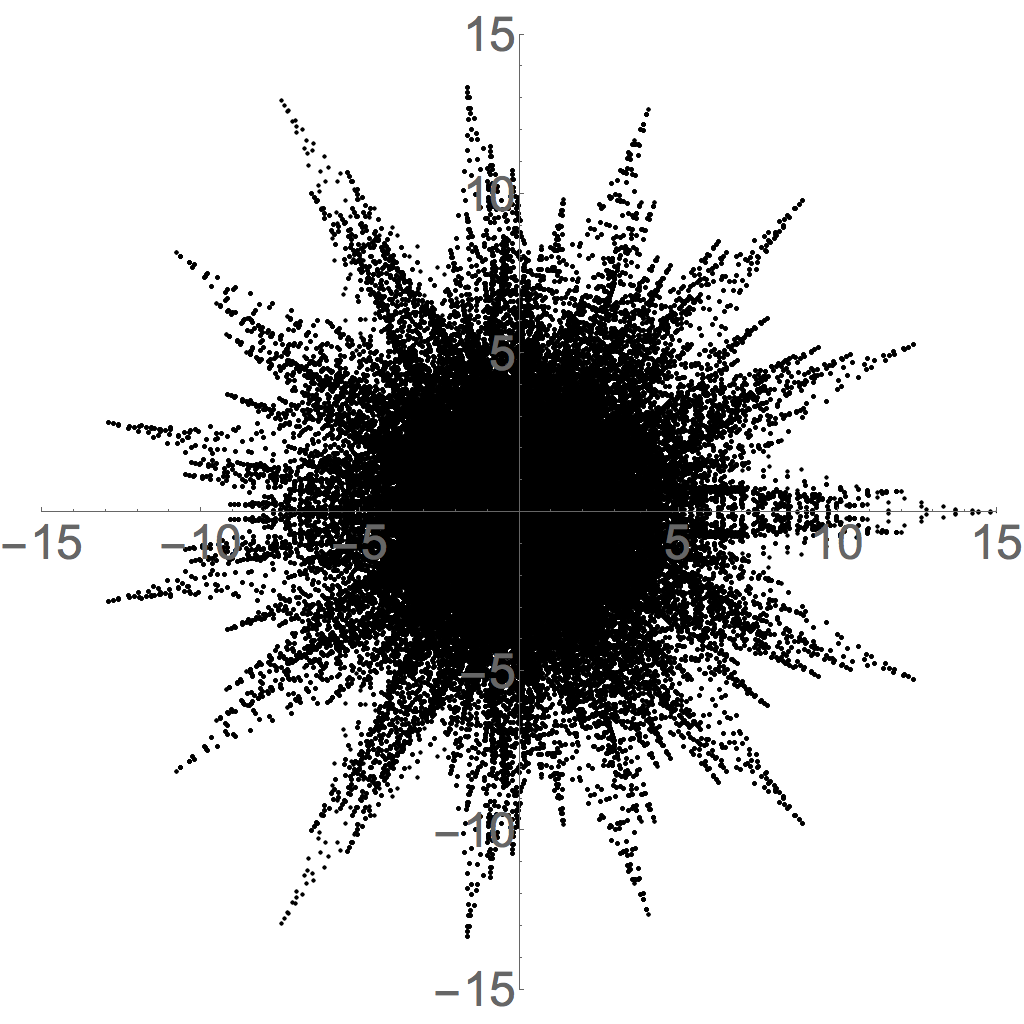}
		\caption{\scriptsize $(1033,1031,219191)$}
		\label{mpfig1}
	\end{subfigure}
	\quad
	\begin{subfigure}[b]{0.3\textwidth}
		\includegraphics[width=\textwidth]{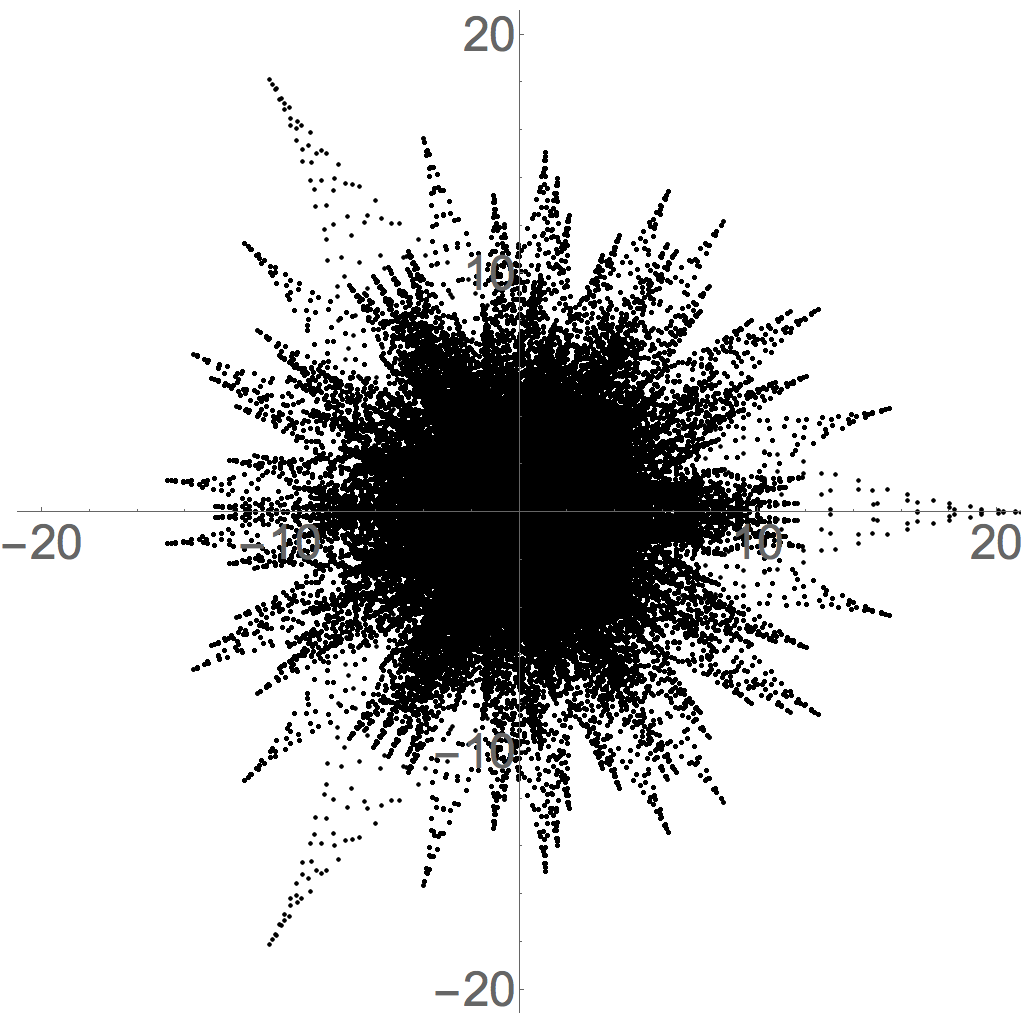}
		\caption{\scriptsize $(1153,1163,120562)$}
		\label{mpfig2}
	\end{subfigure}
	\quad
	\begin{subfigure}[b]{0.3\textwidth}
		\includegraphics[width=\textwidth]{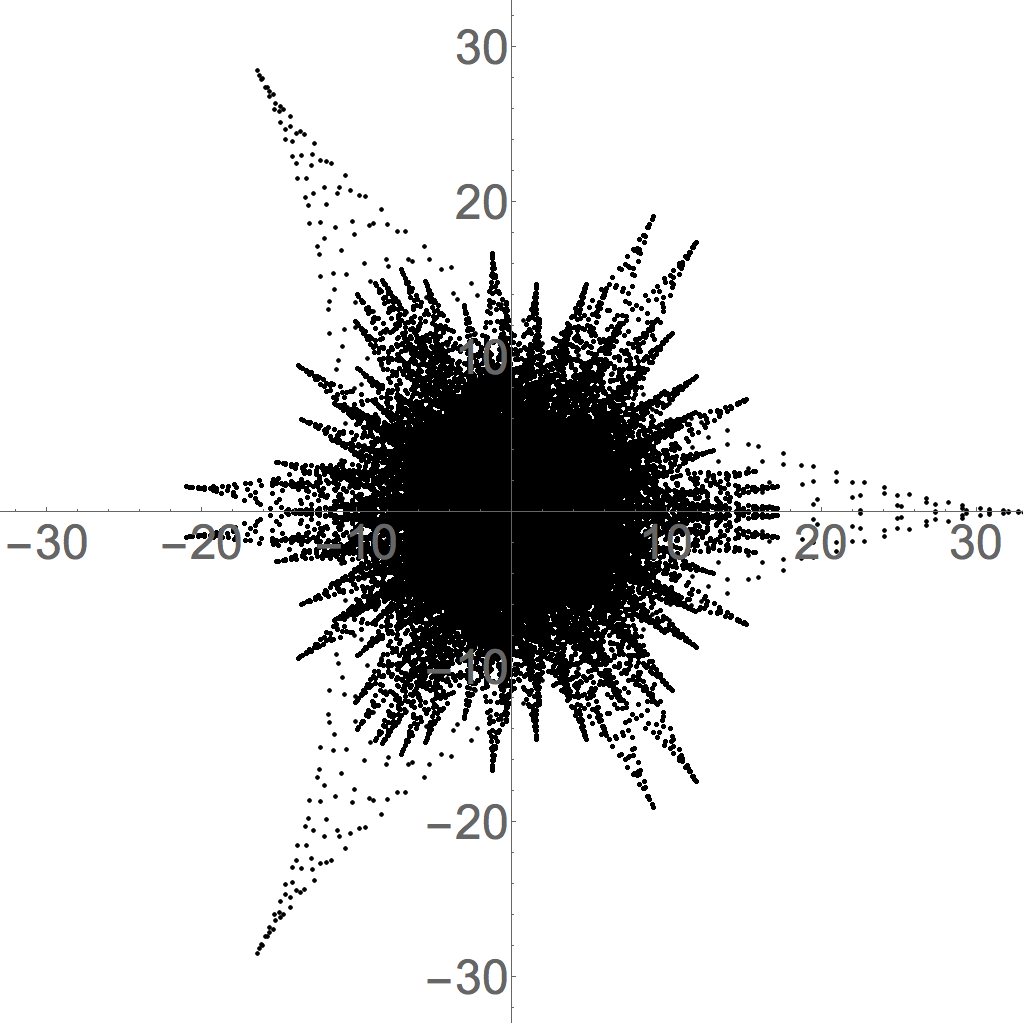}
		\caption{\scriptsize $(1399,1409,240237)$}
		\label{mpfig3}
	\end{subfigure}
	\caption{For the given triples $(m,n,\omega)$, the values of the cyclic supercharacters $\sigma_\omega$ mod $mn$ belong to $H_3\otimes H_{\prm_2}$ where, from left to right, $r_2=5$, $7$ and $11$.}
	\label{mpfig}
\end{figure}

There is no obvious characterization of $H_a\otimes H_b$, such as a parametrization of its boundary, even in terms of parametrizations of the boundaries of $H_a$ and $H_b$. We can, however, give a concrete description of the boundary of the Minkowski product of two polygons that does not appear to have been recorded previously. We defer the proof, an application of \cite[Theorem 2.4]{li2015}, to the Appendix. 

\begin{prop}
	For odd primes $k<\ell$, the boundary of $P_k\otimes P_\prm$ is contained in the $k\prm$-fold dihedral closure of the union of line segments connecting $k\prm e(\frac{1}{k}-\frac{1}{\prm})$ to $k\prm e(\frac{1}{k}+\frac{1}{\prm})$ and $k\prm\cos(\frac{\pi}{k})/\cos(\frac{\pi}{\prm})$.
	\label{dihprop}
\end{prop}

\subsection{Gauss sums}

Henceforth, $p$ will denote an odd prime number. Recall that a \emph{character} mod $p$ is a group map $\chi:(\Z/p\Z)^\times \to \T$. For each integer $k$, let $\chi^k$ be the character $x\mapsto \chi(x)^k$, and recall that the \emph{order} of $\chi$ is the smallest positive $k$ for which $\chi^k$ is identically 1. For each $p$, the unique character mod $p$ of order 2 is the familiar Legendre symbol. 

There are two types of exponential sum bearing the name \emph{Gauss sum mod $p$ of order $k$}, which we distinguish by their notation. The first, defined for any positive divisor $k$ of $\phi(p)$, is the function $g_k:(\Z/p\Z)^\times\to \C$ given by
\begin{equation*}
g_k(t)=\sum_{j=1}^{p} e\left(\frac{tj^k}{p}\right).
\end{equation*}
The similarity in notation to the functions in Theorem \ref{dukethm} is pure coincidence; the reader may consider the notation overwritten. For all $t$ coprime to $p$, notice that
\begin{equation}
g_1(t)=\sum_{j=0}^{p-1} e(t/p)^j=\frac{1-e(t/p)^p}{1-e(t/p)} = 0.
\label{g1eq}
\end{equation}

The second type of \emph{Gauss sum mod $p$ of order $k$}, defined in terms of a character $\chi$ mod $p$ of order $k$, is also a function $G(\cdot,\chi) : (\Z/p\Z)^\times \to \C$, this time given by
\begin{equation*}
G(t,\chi)=\sum_{j=1}^{p-1} \chi(j) e\left(\frac{tj}{p}\right).
\end{equation*}
We write $G(\chi)=G(1,\chi)$ and make tacit use of the following identities:
\begin{equation*}
G(t,\chi)=\overline{\chi(t)}G(\chi)=\chi(-1)\overline{G(t,\overline{\chi})}.
\end{equation*}
The two types of Gauss sum are related by
\begin{equation}
g_k(t)=\sum_{j=1}^{k-1} G(t,\chi^j).
\label{sumreleq}
\end{equation}

In addition to proofs of the last few facts, the reader can find in \cite{berndt1998} explicit evaluations of $g_k$ for small $k$ up to certain sign ambiguities, some of which persist to this day. Gauss resolved the issue for $g_2$ in terms of the Legendre symbol $\chi$ by showing that
\begin{equation}
\chi(t)g_2(t)=
\begin{cases}
\sqrt{p}, &\mbox{if }p\equiv 1\pmod{4}\\
i\sqrt{p}, &\mbox{if }p\equiv 3\pmod{4}.\\
\end{cases}
\label{quadsumeq}
\end{equation}
The next two lemmas are of technical import only; the casual reader is invited to skim their proofs, although they are used in what follows. We denote the real part of a complex number $z$ by $\Re(z)$ and the imaginary part by $\Im(z)$.

\begin{lem}
If, in addition to the hypotheses of Theorem \ref{splitthm}, $m=p$ is an odd prime, $\omega_p$ is a primitive root mod $p$, and $t$ a unit mod $p$, then
\begin{equation}
\sigma_\omega(sm+tn) = \frac{1}{(u,\phi(p))}\sum_{j=1}^u (g_{(u,\phi(p))}(\omega^j t)-1)\sigma_{\omega_n^u}(\omega^j s).
\label{periodsumeq}
\end{equation}

\begin{proof}
To Theorem \ref{splitthm}, apply the observation that
\begin{equation*}g_k(r)-1=k\sum_{j=1}^{\phi(p)/k} e\left(\frac{r\omega^{jk}}{p}\right)=k\sigma_{\omega^k}(r).\qedhere\end{equation*}
\end{proof}
\label{fracglem}
\end{lem}

\begin{lem}
If $k$ is a positive even integer and $p\equiv 1\pmod{2k}$ is an odd prime, then $g_k$ is real valued.

\begin{proof}
Let $\chi$ be a character mod $p$ of order $k$. We have
\begin{align*}
g_k(t)
&=\sum_{j=1}^{k-1} G(t,\chi^j)\\
&=G(t,\chi^{k/2})+\sum_{j=1}^{k/2-1} G(t,\chi^j) + \sum_{j=1}^{k/2-1} G(t,\overline{\chi}^j)\\
&=g_2(t)+\sum_{j=1}^{k/2-1} G(t,\chi^j) + \sum_{j=1}^{k/2-1} \chi^j(-1)G(t,\overline{\chi}^j)\\
&=g_2(t)+\sum_{j=1}^{k/2-1} G(t,\chi^j) + \sum_{j=1}^{k/2-1} (-1)^{j\phi(p)/u} G(t,\overline{\chi}^j)\\
&=g_2(t)+2\sum_{j=1}^{k/2-1} \Re(G(t,\chi^j)),
\end{align*}
where $g_2(t)$ is real by \eqref{quadsumeq}.
\end{proof}
\label{realglem}
\end{lem}

\subsection{Main results}
\label{mainsec}

For the rest of the article, it will suit us to treat $\C$ as an $\R$-algebra with basis $(1,i)$, so that if $z=\alpha+i\beta$ for real $\alpha$ and $\beta$, then
\begin{equation*}
\begin{pmatrix}
a&b\\
c&d
\end{pmatrix}
z=
\begin{pmatrix}
a&b\\
c&d
\end{pmatrix}
\begin{pmatrix}
\alpha\\
\beta
\end{pmatrix}
=
\begin{pmatrix}
a\alpha+b\beta\\
c\alpha+d\beta
\end{pmatrix}
=(a\alpha+b\beta)+i(c\alpha+d\beta).
\end{equation*}
The following results are typical consequences of Lemma \ref{fracglem}. By exploiting the additional requirement that $\omega_n$ be a root of $-1$, we are able to write $\sigma_\omega$ in terms of $\sigma_{\omega_n^u}$ subject to certain $\R$-linear transformations. When the corresponding matrix representations have at most $2$ nonzero entries, we obtain explanations of various graphical features, including some depicted in \cite{duke2015} and \cite{garcia2015}, which is our goal. Ellipses, rhombi, astroids, and other plane figures lurk in the images of the cyclic supercharacters described by Theorem \ref{shapethm}. We present several examples in the next section. 

\begin{thm}
In the notation of Theorem \ref{splitthm}, suppose that $u$ is even, $v$ odd, and $m=p$ an odd prime. Let $r$ be a positive integer, and suppose further that $\omega_n^{uv/2}=-1$, $\ord(\omega_p)=\frac{1}{r}\phi(p)$ and $(t,p)=1$.

\begin{enumerate}[label={(\alph*)},ref={\thethm(\alph*)}]
\item If $p\equiv 1\pmod{2r u}$, then
\begin{equation*}
\sigma_\omega(sp+tn)=\frac{2}{r u}\sum_{j=1}^{u/2}
\begin{pmatrix}
g_{r u/2}(\omega^j t)-1&0\\
0&g_{r u}(\omega^j t)-g_{r u/2}(\omega^j t)
\end{pmatrix}
\sigma_{\omega_n^u}(\omega^j s).
\end{equation*}
\label{ellipsethm}

\item If $4|u$ and $p\equiv 1+\frac{r u}{2}\pmod{r u}$, then
\begin{equation*}
\sigma_\omega(sp+tn)=\frac{4}{r u}\sum_{j=1}^{u/2}
\begin{pmatrix}
\Re(g_{r u/2}(\omega^j t))-1&0\\
\Im(g_{r u/2}(\omega^j t))&0
\end{pmatrix}
\sigma_{\omega_n^u}(\omega^j s).
\end{equation*}
\label{rhomthm}
\end{enumerate}

\begin{proof}
We show (a) in detail and describe an analogous proof of (b). In either setting, since $v$ is odd, the set of residues of the form $\omega_n^{u/2}\omega^{uj}_n$ for $j=1,\ldots,v$ is equal to the set of residues of the form $\omega_n^{uv/2}\omega_n^{uj}=-\omega_n^{uj}$. Hence
\begin{equation}
\sigma_{\omega_n^u}(\omega^{u/2}s)
=\sum_{j=1}^v e\left(\frac{\omega^{u/2}\omega^{uj}}{n}\right)
=\sum_{j=1}^v e\left(\frac{-\omega^{uj}}{n}\right)
=\overline{\sigma_{\omega_n^u}(s)}
\label{randeq}
\end{equation}
for all $s$. Suppose now that $p\equiv 1\pmod{2r u}$, as in (i). Lemma \ref{fracglem} says that
\begin{equation*}
\sigma_\omega(sp+tn)=\frac{1}{r u}\sum_{j=1}^{uv} (g_{r u}(\omega^jt)-1)\sigma_{\omega_n^u}(\omega^j s),
\end{equation*}
where, by \eqref{randeq} and Lemma \ref{realglem}, we have
\begin{multline}
(g_{r u}(\omega^j t)-1)\sigma_{\omega_n^u}(\omega^j s)
+(g_{r u}(\omega^{j+u/2} t)-1)\sigma_{\omega_n^u}(\omega^{j+u/2} s)\\
=
\begin{pmatrix}
g_{r u}(\omega^j t)+g_{r u}(\omega^{j+u/2} t)-2&0\\
0&g_{r u}(\omega^j t)-g_{r u}(\omega^{j+u/2} t)
\end{pmatrix}
\sigma_{\omega_n^u}(\omega^j s),
\label{bigmatrixeq}
\end{multline}
for $j=1,\ldots,\frac{u}{2}$. Let $\chi$ be the character mod $p$ of order $r u$ with $\chi(\omega)=e(\frac{1}{u})$, and notice that
\begin{equation*}
g_{r u}(\omega^{u/2} t)
= \sum_{j=1}^{r u-1} G(\omega^{u/2}t,\chi^j)
=\sum_{j=1}^{r u-1} \chi^j(\omega^{-u/2})G(t,\chi^j)
= \sum_{j=1}^{r u-1} (-1)^j G(t,\chi^j).
\end{equation*}
It follows that
\begin{equation}
g_{r u}(t)+g_{r u}(\omega^{u/2} t) = \sum_{j=1}^{r u - 1} G(t,\chi^j)+(-1)^jG(t,\chi^j)=2g_{r u/2}(t),
\label{topeq}
\end{equation}
which gives
\begin{equation*}
g_{r u}(t)-g_{r u}(\omega^{u/2} t) = 2(g_{r u}(t)-g_{r u/2}(t)).
\end{equation*}
Combining this with \eqref{bigmatrixeq} and \eqref{topeq} completes the proof of (a). The argument for (b) is similar in spirit to the one just given, with the main differences being that $\sigma_{\omega_p^u}=g_{r u/2}$ and $g_{r u/2}(\omega^{u/2} t) = g_{r u/2}(t)$ for all $t$.
\end{proof}
\label{shapethm}
\end{thm}

Certain families of real-valued cyclic supercharacters, while less interesting from a visual standpoint, can also be described by Theorem \ref{splitthm}. The following proposition describes two. We omit the proof, which resembles the previous one.

\begin{prop}
Suppose, in addition to the hypotheses of Lemma \ref{fracglem}, that $v$ is odd and $\omega_n^{uv/2}=-1$.
\begin{enumerate}[label={(\alph*)},ref={\thethm(\alph*)}]
\item
If $u=2$ and $p\equiv 3\pmod{4}$, then
\begin{equation*}
\sigma_\omega(sp+tn)
=
\begin{pmatrix}
-1&-\chi(t)\sqrt{p}\\
0&0
\end{pmatrix}
\sigma_{\omega_n^2}(s).\end{equation*}
\item
Suppose that $p\equiv 5\pmod{8}$, and let $\chi$ be the unique character mod $p$ with $\chi(\omega)=i$. If $u=4$, then
\begin{equation*}
\sigma(sp+tn)
=\begin{pmatrix}
\frac{1}{2}(g_2(t)-1)&\Re(G(t,\chi))\\
0&0
\end{pmatrix}
\sigma_{\omega_n^4}(s)
+\begin{pmatrix}
\frac{1}{2}(-g_2(t)-1)&\Im(G(t,\chi))\\
0&0
\end{pmatrix}
\sigma_{\omega_n^4}(\omega s).
\end{equation*}
\end{enumerate}
\end{prop}

\section{Examples}

The images of cyclic supercharacters $\sigma_\omega$ satisfying the hypotheses of Theorem \ref{shapethm} belong to Minkowski sums of $\im(\sigma_{\omega_n})$ where each summand is subject to an $\R$-linear transformation. This observation informs our perspective in what follows.

For a positive integer $a$, a divisor $b$ of $a$, and a unit $\omega$ mod $a$, the sets
\begin{equation*}
\{\sigma_\omega(y) : y\equiv j \pmod{b}\}
\end{equation*}
for $j=0,1,\ldots,b-1$ are called the \emph{layers mod $b$} of $\sigma_\omega$. The layer mod $b$ corresponding to $j=0$ is called \emph{trivial}. Different shades of points plotted in Figures \ref{prettyfig}, \ref{stretchhypofig2}, \ref{windingfig2}, \ref{rhomfig2}, \ref{mysteryfig1} and \ref{mysteryfig2} mark distinct layers of the corresponding cyclic supercharacters. That is, in each figure, if $y\equiv y'\pmod b$ for some fixed divisor $b$ of the modulus, then $\sigma_\omega(y)$ and $\sigma_\omega(y')$ have the same shade. Under the hypotheses of Theorem \ref{shapethm}, $\sigma_\omega$ has $r+1$ distinct layers mod $p$, the trivial one of which is the subset of $\R$ consisting of all values $\sigma_\omega(sp+tn)$ for which $p|t$.

\subsection{Stretching}
\label{ellipsesec}

In the following, we assume the hypotheses of Theorem \ref{ellipsethm}, where the $\R$-linear transformations discussed above are scalings along the real and imaginary axes, possibly by negative factors. When $n$ is an odd prime distinct from $p$ and $v$ is a power of an odd prime, the discussion in Section \ref{ppmodsec} tells us that the corresponding images $\im(\omega_\omega)$ can be arranged in sequences filling out Minkowski sums of stretched versions of $H_r$. Figure \ref{stretchhypofig} illustrates this behavior.

\begin{figure}[ht]
\begin{subfigure}[b]{0.3\textwidth}
\includegraphics[width=\textwidth]{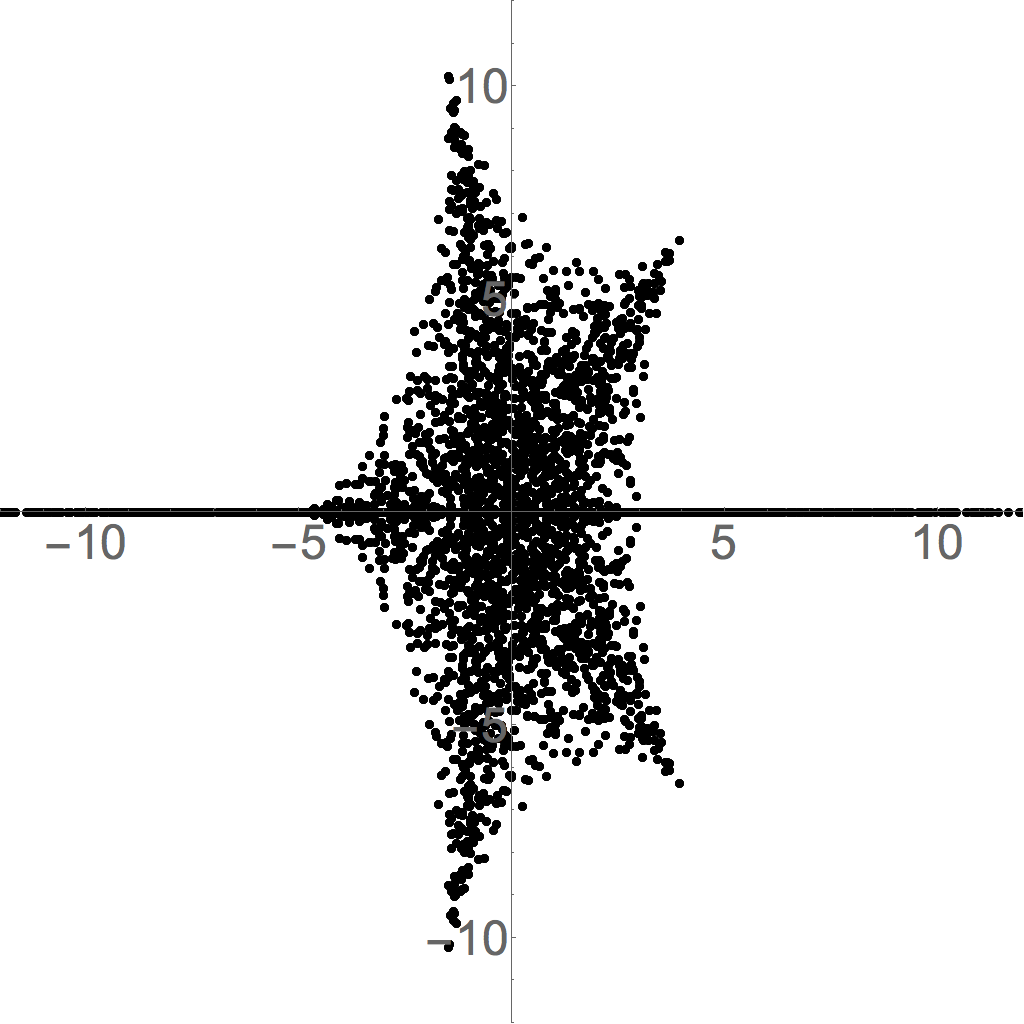}
\caption{\scriptsize $(5,13291,8142)$}
\end{subfigure}
\quad
\begin{subfigure}[b]{0.3\textwidth}
\includegraphics[width=\textwidth]{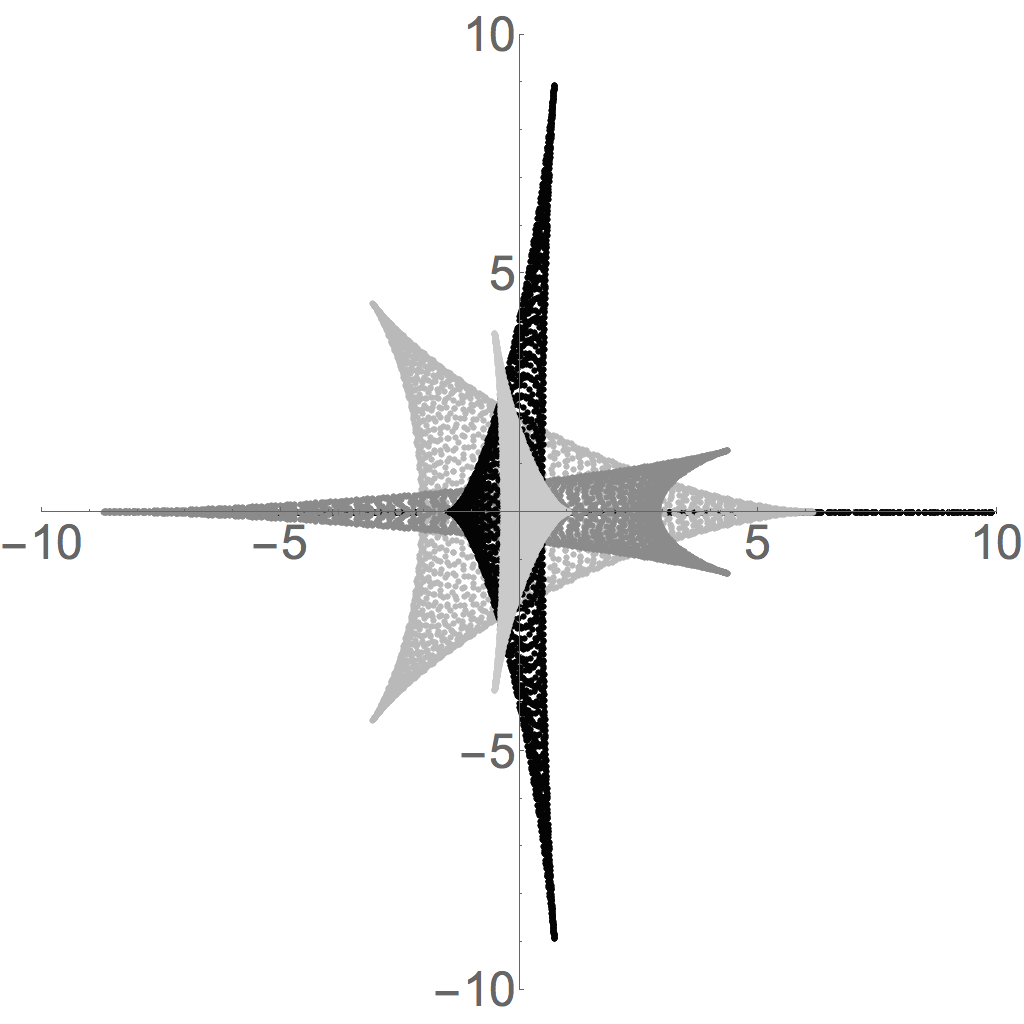}
\caption{\scriptsize $(17,6493,27213)$}
\label{stretchhypofig2}
\end{subfigure}
\quad
\begin{subfigure}[b]{0.3\textwidth}
\includegraphics[width=\textwidth]{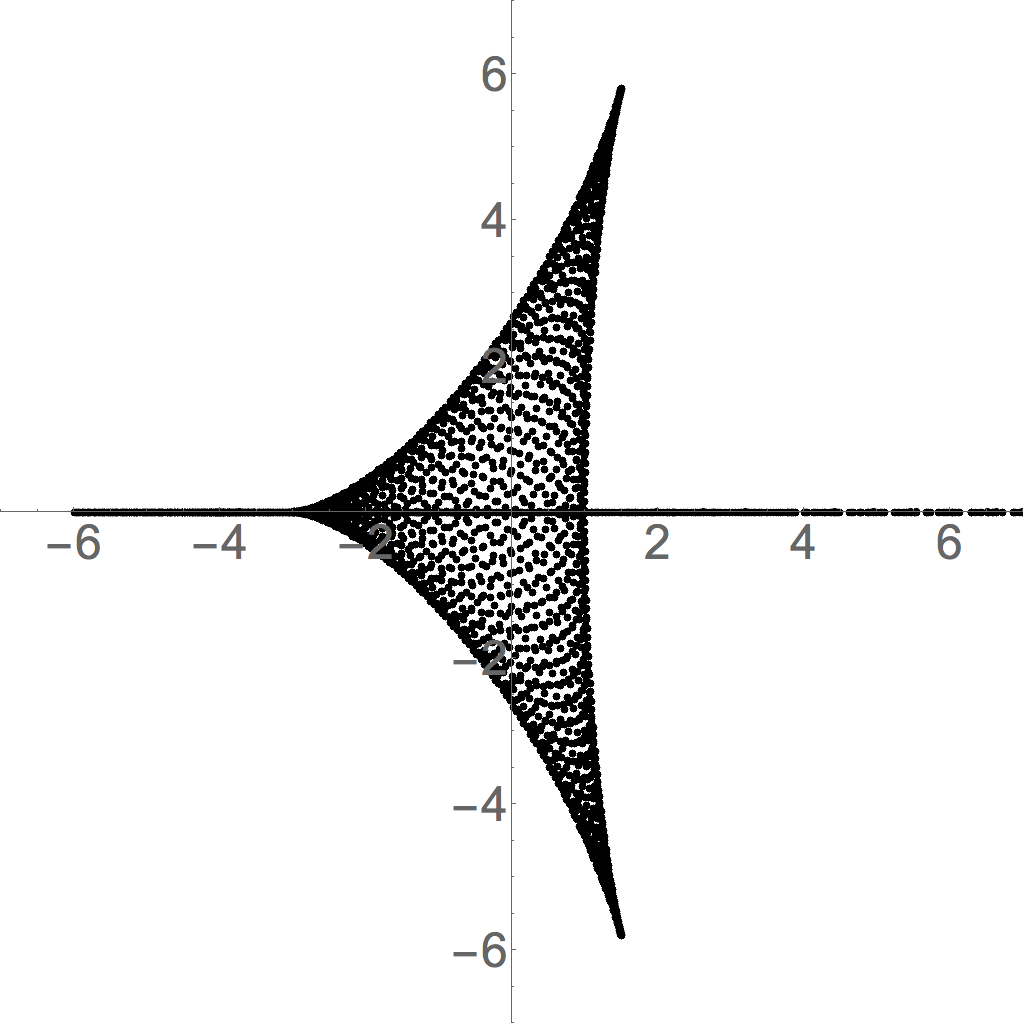}
\caption{\scriptsize $(5,6247,2317)$}
\end{subfigure}
\caption{For the given triples $(p,n,\omega)$, the images of the cyclic supercharacters $\sigma_\omega$ mod $pn$ are explained by Theorem \ref{ellipsethm}. See Section \ref{ellipsesec} for details.}
\label{stretchhypofig}
\end{figure}

On the other hand, if $v=1$ for any $n$, ellipses emerge. For the remainder of the subsection, suppose that $v=1$. We see that
\begin{equation*}
\sigma_\omega(sp+tn)
=\frac{2}{r u}\sum_{j=1}^{u/2}
\begin{pmatrix}
g_{r u/2}(\omega^j t)-1&0\\
g_{r u}(\omega^j t)-g_{r u/2}(\omega^j t)&0
\end{pmatrix}
e\left(\frac{s}{n}\right).
\end{equation*}
Here, $\sigma_\omega(sp+tn)$ belongs to a Minkowski sum of ellipses in standard form:
\begin{equation*}
\bigoplus_{j=1}^{u/2} \left\{z\in \C : \frac{\Re(z)^2}{(g_{r u/2}(\omega^j t)-1)^2}
+ \frac{\Im(z)^2}{(g_{r u}(\omega^j t)-g_{r u/2}(\omega^j t))^2} = 1\right\}.
\end{equation*}

Take, for example, the case $u=2$ and $r=1$. In this situation, the nontrivial layer of $\sigma_\omega$ mod $p$ is contained in the ellipse with equation $\Re(z)^2+\Im(z)^2/p=1$. This behavior, depicted by Figures \ref{ellipsefig1} and \ref{ellipsefig3}, was first noted in \cite[Proposition 5.2]{duke2015}, but the framework here is more general. In particular, it is apparent now that such examples are more common than previously thought. Figure \ref{ellipsefig2} illustrates the case $u=r=2$ and, accordingly, features $r=2$ distinct ellipses.

\begin{figure}[ht]
\begin{subfigure}[b]{0.3\textwidth}
\includegraphics[width=\textwidth]{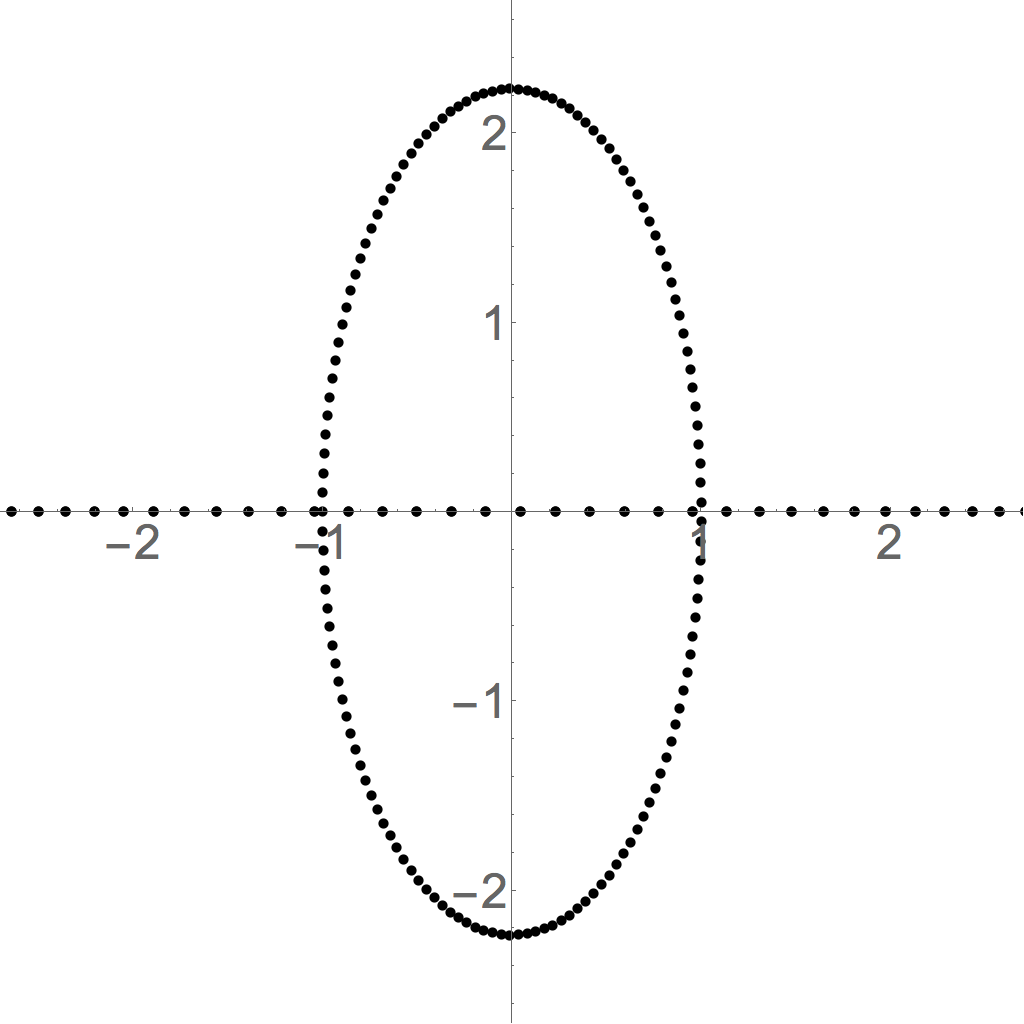}
\caption{\scriptsize $(5,137,273)$}
\label{ellipsefig1}
\end{subfigure}
\quad
\begin{subfigure}[b]{0.3\textwidth}
\includegraphics[width=\textwidth]{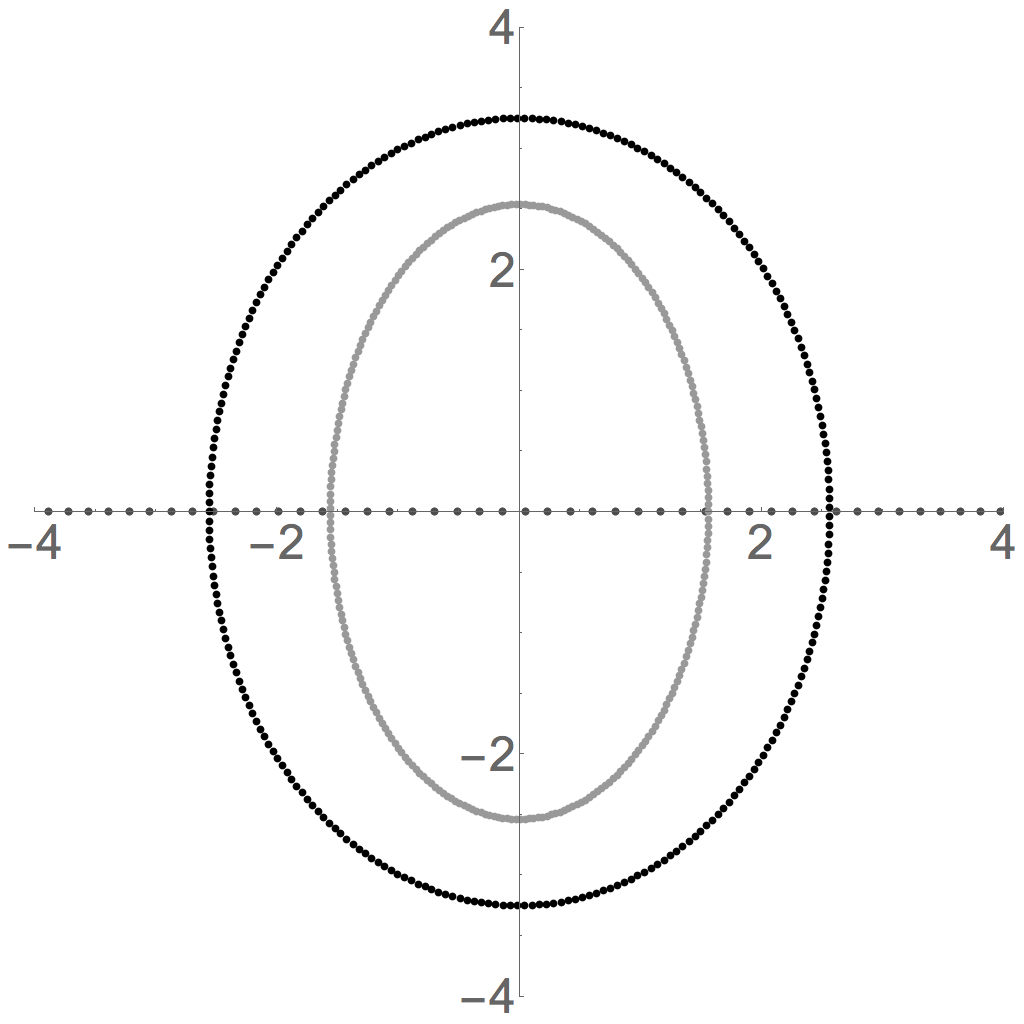}
\caption{\scriptsize $(17,269,1613)$}
\label{ellipsefig2}
\end{subfigure}
\quad
\begin{subfigure}[b]{0.3\textwidth}
\includegraphics[width=\textwidth]{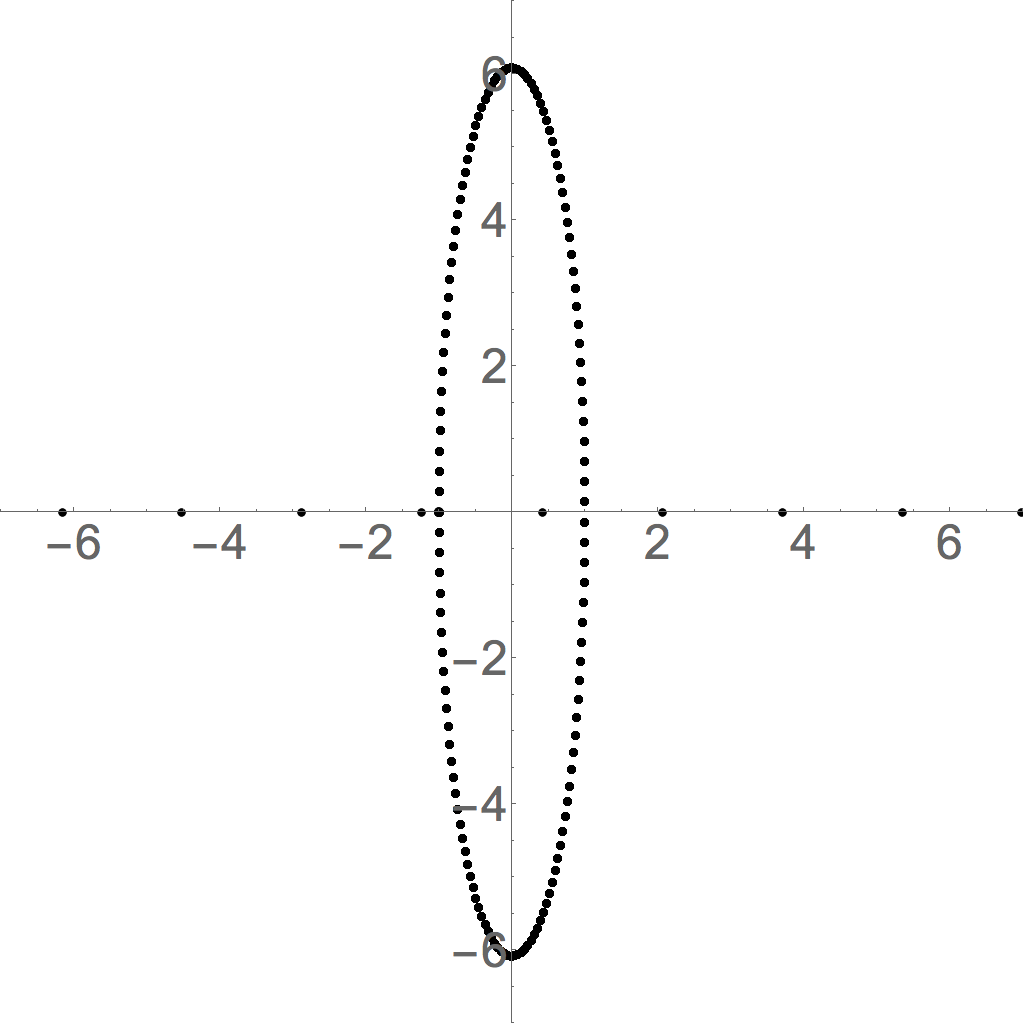}
\caption{\scriptsize $(37,137,684)$}
\label{ellipsefig3}
\end{subfigure}
\caption{For the given triples $(p,n,\omega)$, discretized versions of ellipses appear in the plots of cyclic supercharacters $\sigma_\omega$ mod $pn$. See Section \ref{ellipsesec} for details.}
\label{ellipsefig}
\end{figure}

Suppose now that $u=4$ and $r=1$, and that $\chi$ is a character mod $p$ of order 4. It can be shown that each nontrivial layer of $\sigma_\omega$ mod $p$ is contained in the image of the $\lcm(2,n)$-th roots of unity under the map
\begin{equation*}
z\mapsto \begin{pmatrix}\frac{1}{2} (\sqrt{p}-1)&0\\0&\Re(G(\chi))\end{pmatrix}z + \begin{pmatrix}\frac{1}{2} (\sqrt{p}+1)&0\\0&\Im(G(\chi))\end{pmatrix}z^\omega,
\end{equation*}
which can be rewritten to reflect the fact that $|G(\chi)|=\sqrt{p}$. The image in question is most easily visualized as the path of a point winding $\omega$ times around an ellipse whose center travels once around another ellipse. Figures \ref{windingfig1} and \ref{windingfig3} depict this behavior, while Figure \ref{windingfig2}, which appeared in \cite{duke2015} unexplained, illustrates the case $u=4$ and $r=2$.

\begin{figure}[ht]
\begin{subfigure}[b]{0.3\textwidth}
\includegraphics[width=\textwidth]{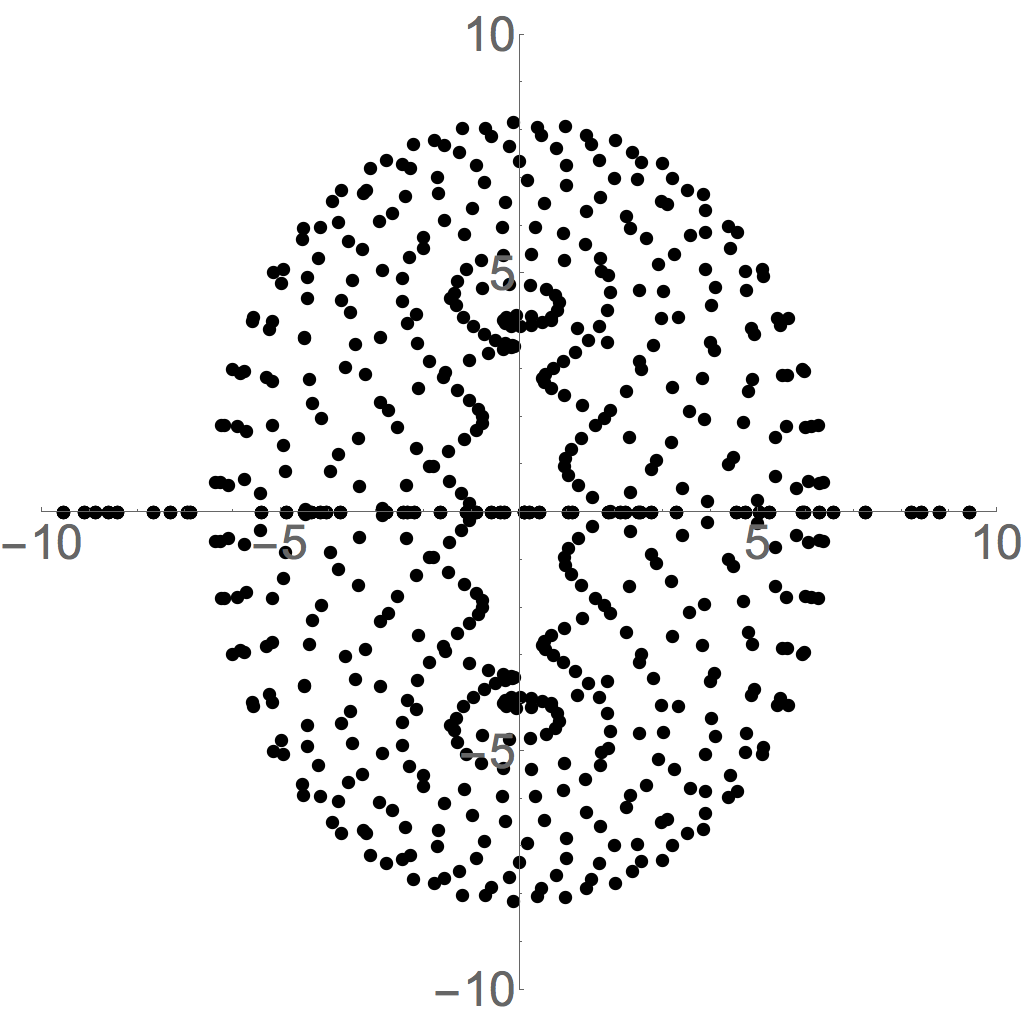}
\caption{\scriptsize $(41,541,52)$}
\label{windingfig1}
\end{subfigure}
\quad
\begin{subfigure}[b]{0.3\textwidth}
\includegraphics[width=\textwidth]{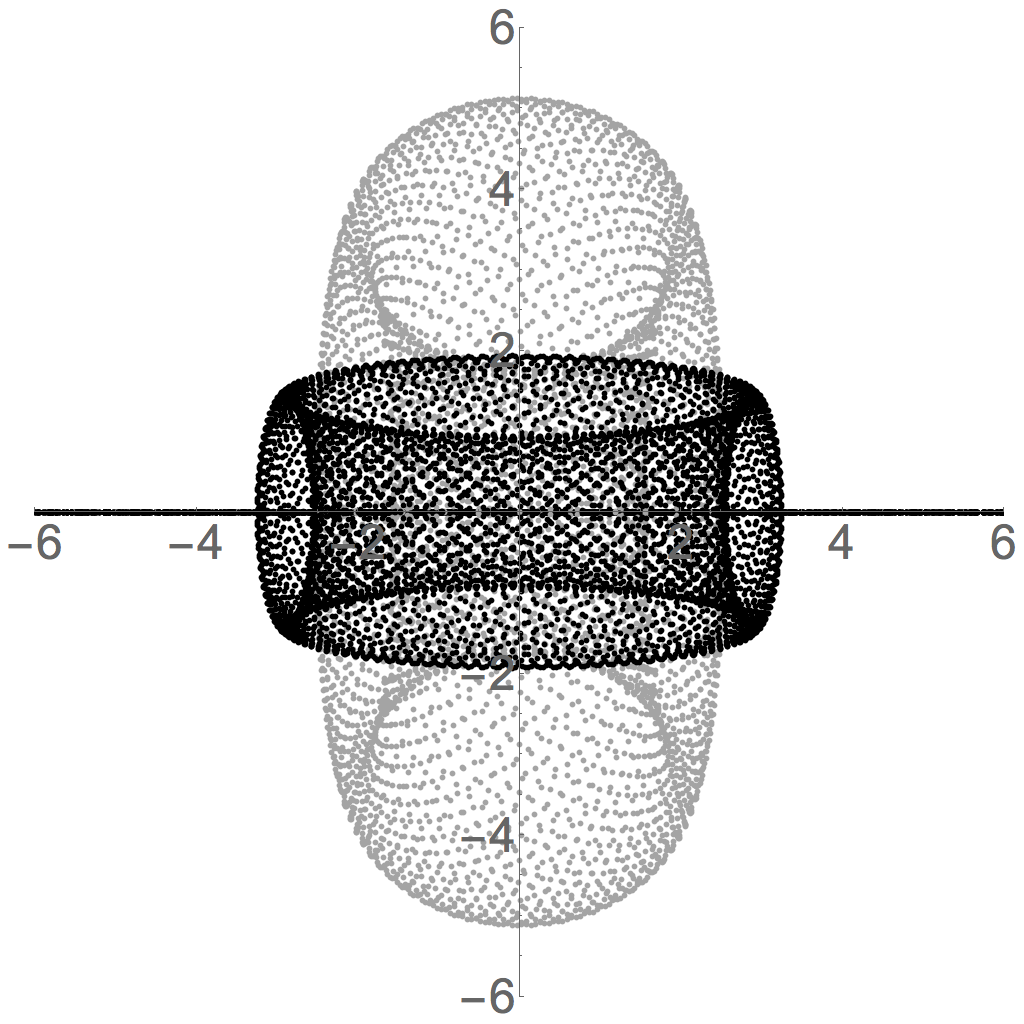}
\caption{\scriptsize $(17,5365,2337)$}
\label{windingfig2}
\end{subfigure}
\quad
\begin{subfigure}[b]{0.3\textwidth}
\includegraphics[width=\textwidth]{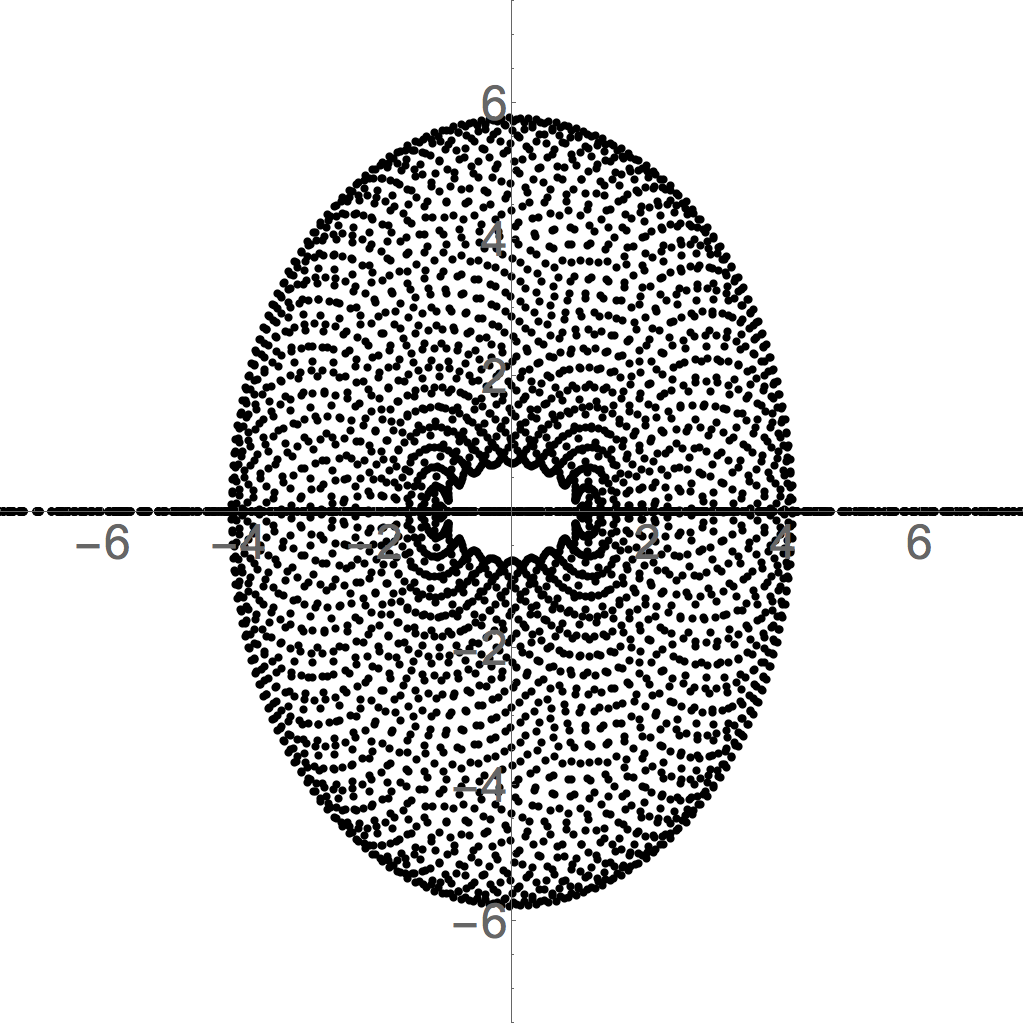}
\caption{\scriptsize $(17,3581,364)$}
\label{windingfig3}
\end{subfigure}
\caption{For the given triples $(p,n,\omega)$, the ovate figures contained in the plots of the cyclic supercharacters $\sigma_\omega$ mod $pn$ are the effect of one ellipse ``winding around'' another. See Section \ref{ellipsesec} for details.}
\end{figure}

Returning to the case $u=2$, suppose now that $r$ is maximal, i.e., $r = \frac{\phi(p)}{4}$. Each of the $r$ nontrivial layers of $\sigma_\omega$ is the image of the set of $\lcm(2,n)$-th roots of unity under the $\R$-linear map with matrix
\begin{equation*}
\begin{pmatrix}
\cos(2\pi t/n)+\cos(2\pi \omega t/n)&0\\
0&\cos(2\pi t/n)-\cos(2\pi \omega t/n)
\end{pmatrix},
\end{equation*}
for some $t$ coprime to $p$. This image, in turn, belongs to an ellipse whose semimajor and semiminor axes sum to at most 4. The envelope of the family of all such ellipses is the boundary of $H_4$. Accordingly, for large $p$, plots of these cyclic supercharacters tend to resemble $H_4$. Figure \ref{astroidfig} presents several examples.

\begin{figure}[ht]
\begin{subfigure}[b]{0.3\textwidth}
\includegraphics[width=\textwidth]{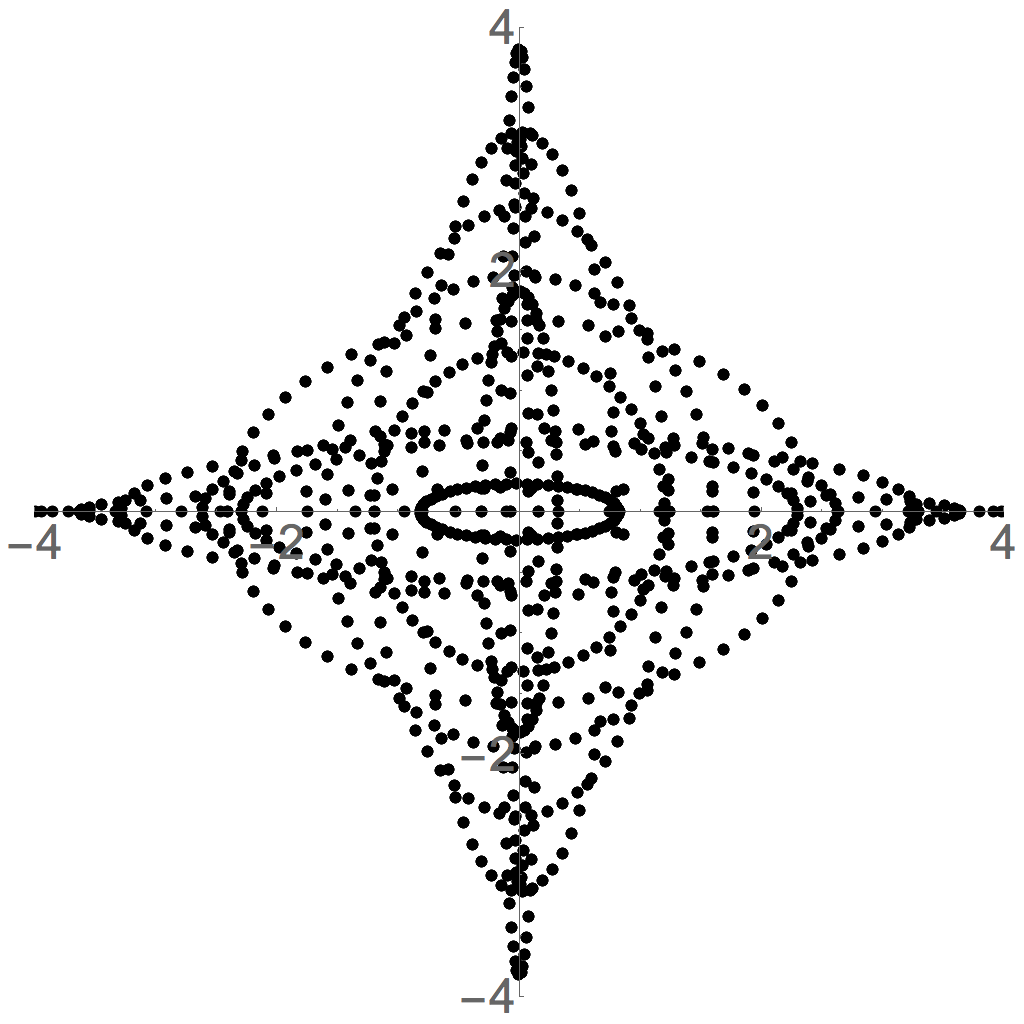}
\caption{\scriptsize $(59,53,235)$}
\end{subfigure}
\quad
\begin{subfigure}[b]{0.3\textwidth}
\includegraphics[width=\textwidth]{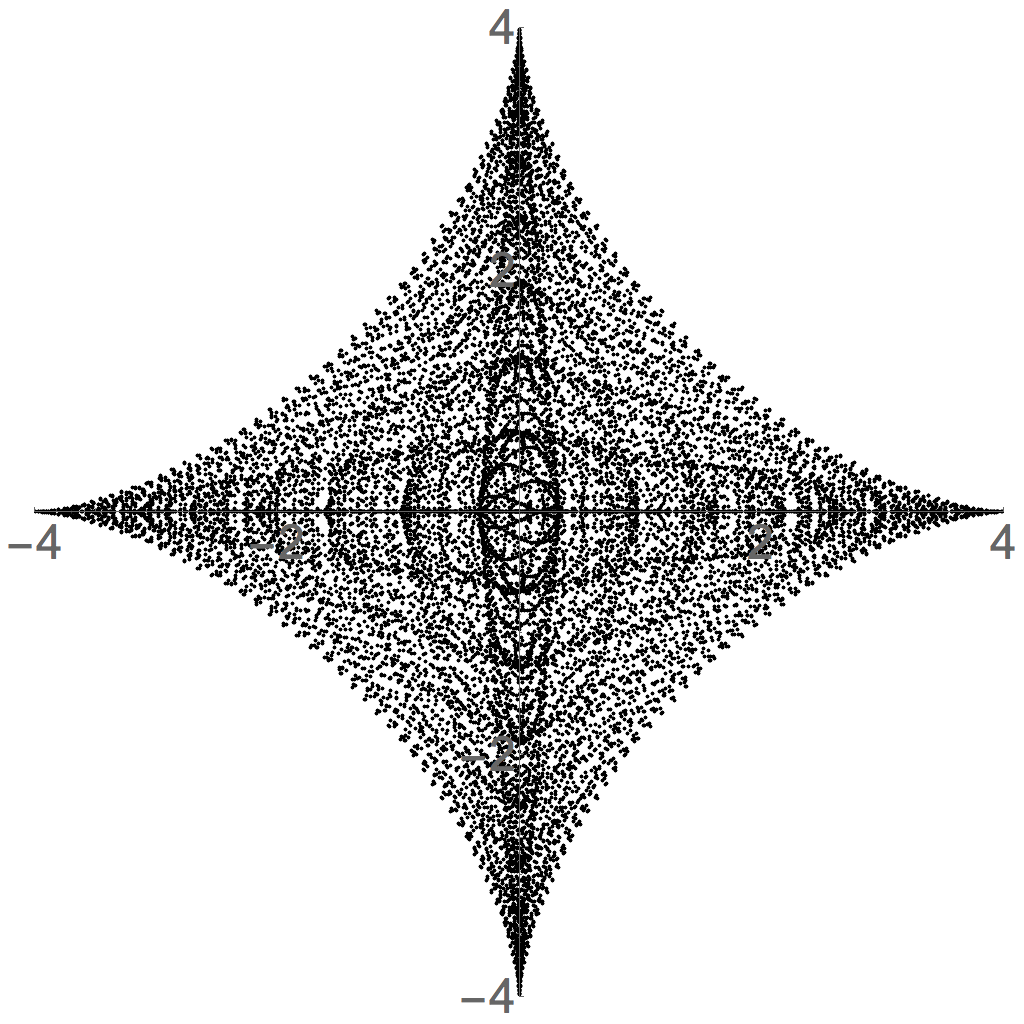}
\caption{\scriptsize $(19,3617,1234)$}
\end{subfigure}
\quad
\begin{subfigure}[b]{0.3\textwidth}
\includegraphics[width=\textwidth]{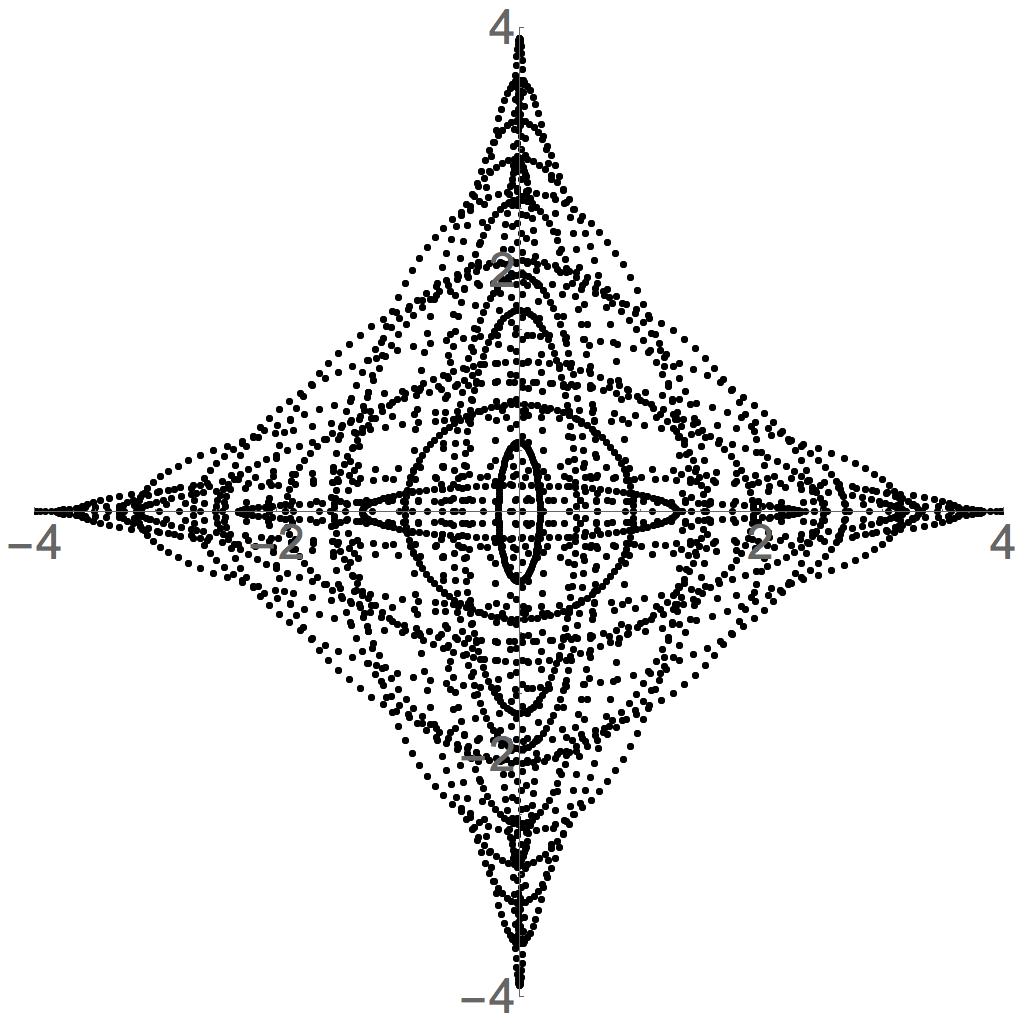}
\caption{\scriptsize $(107,109,1711)$}
\end{subfigure}
\caption{For the given triples $(p,n,\omega)$, the plots of the cyclic supercharacters $\sigma_\omega$ mod $pn$ contain discretized ellipses within $H_4$.}
\label{astroidfig}
\end{figure}

\subsection{Rhombi}
\label{rhomsec}

For this section, we assume the hypotheses of Theorem \ref{rhomthm}, and additionally that $u=4$ and $v=1$. A routine computation gives $g_{2r}(\omega t)=\overline{g_{2r}(t)}$ for all $t$, so
\begin{equation*}
\sigma_\omega(sp+tn)=\frac{1}{r}
\begin{pmatrix}
\Re(g_{2r}(t))-1&0\\
\Im(g_{2r}(t))&0
\end{pmatrix}
e\left(\frac{s}{n}\right)+\frac{1}{r}
\begin{pmatrix}
\Re(g_{2r}(t))-1&0\\
-\Im(g_{2r}(t))&0
\end{pmatrix}
e\left(\frac{\omega s}{n}\right).
\end{equation*}
We claim that scaling the real and imaginary parts of $\sigma_\omega(sp+tn)$ by factors dependent only on $t$ and rotating counterclockwise by $\frac{\pi}{2}$ about the origin yields a point with real and imaginary parts each in the interval $[-1,1]$. Indeed, consider the $\R$-linear map on $\C$ with matrix
\begin{equation*}
T=
\begin{pmatrix}
\frac{\sqrt{2}}{2}&-\frac{\sqrt{2}}{2}\\
\frac{\sqrt{2}}{2}&\frac{\sqrt{2}}{2}
\end{pmatrix}
\begin{pmatrix} 
\frac{\sqrt{2}}{r} (\Re(g_{2r}(t))-1)&0\\
0&\frac{\sqrt{2}}{r}\Im(g_{2r}(t))
\end{pmatrix}^{-1},
\end{equation*}
and notice that
\begin{equation*}
T\sigma_\omega(sp+tn)=
\begin{pmatrix}
1&0\\
0&0
\end{pmatrix}
e\left(\frac{s}{n}\right)
+\begin{pmatrix}
0&0\\
1&0
\end{pmatrix}
e\left(\frac{\omega s}{n}\right).
\end{equation*}
It follows that each nontrivial layer of $\sigma_\omega$ mod $p$ is contained in the convex hull of a rhombus in $\C$ with vertices at $\pm \frac{2}{r} (\Re(g_{2r}(t))-1)$ and $\pm i\frac{2}{r} \Im(g_{2r}(t))$ for some $t$ coprime to $p$. Plots of these cyclic supercharacters appear in both \cite{duke2015} and \cite{garcia2015}. In case $r=1$, as in Figures \ref{rhomfig1} and \ref{rhomfig3}, the vertices of the sole rhombus are at $\pm2$ and $\pm 2i\sqrt{p}$. Figure \ref{rhomfig2} illustrates the case $r=5$.

\begin{figure}[ht]
\begin{subfigure}[b]{0.3\textwidth}
\includegraphics[width=\textwidth]{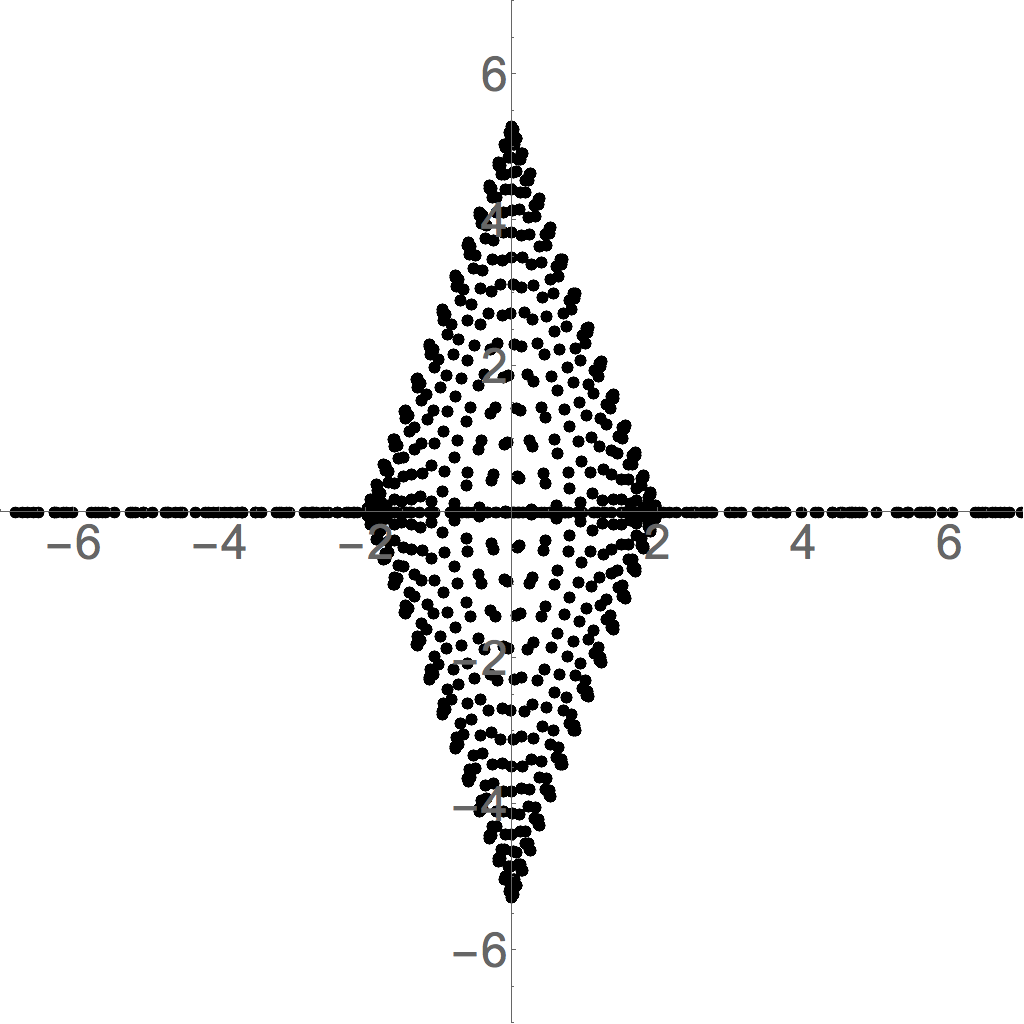}
\caption{\scriptsize $(7,1229,3055)$}
\label{rhomfig1}
\end{subfigure}
\quad
\begin{subfigure}[b]{0.3\textwidth}
\includegraphics[width=\textwidth]{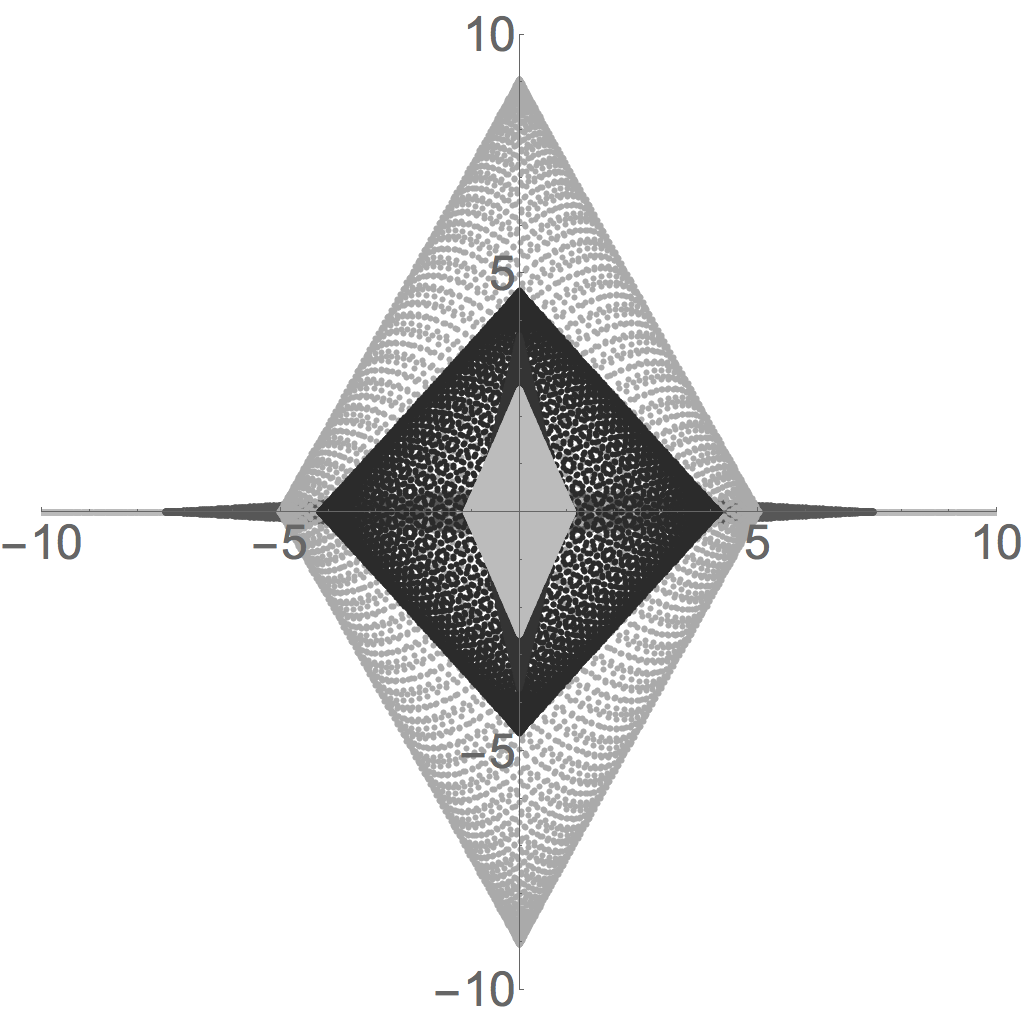}
\caption{\scriptsize $(31,11849,24527)$}
\label{rhomfig2}
\end{subfigure}
\quad
\begin{subfigure}[b]{0.3\textwidth}
\includegraphics[width=\textwidth]{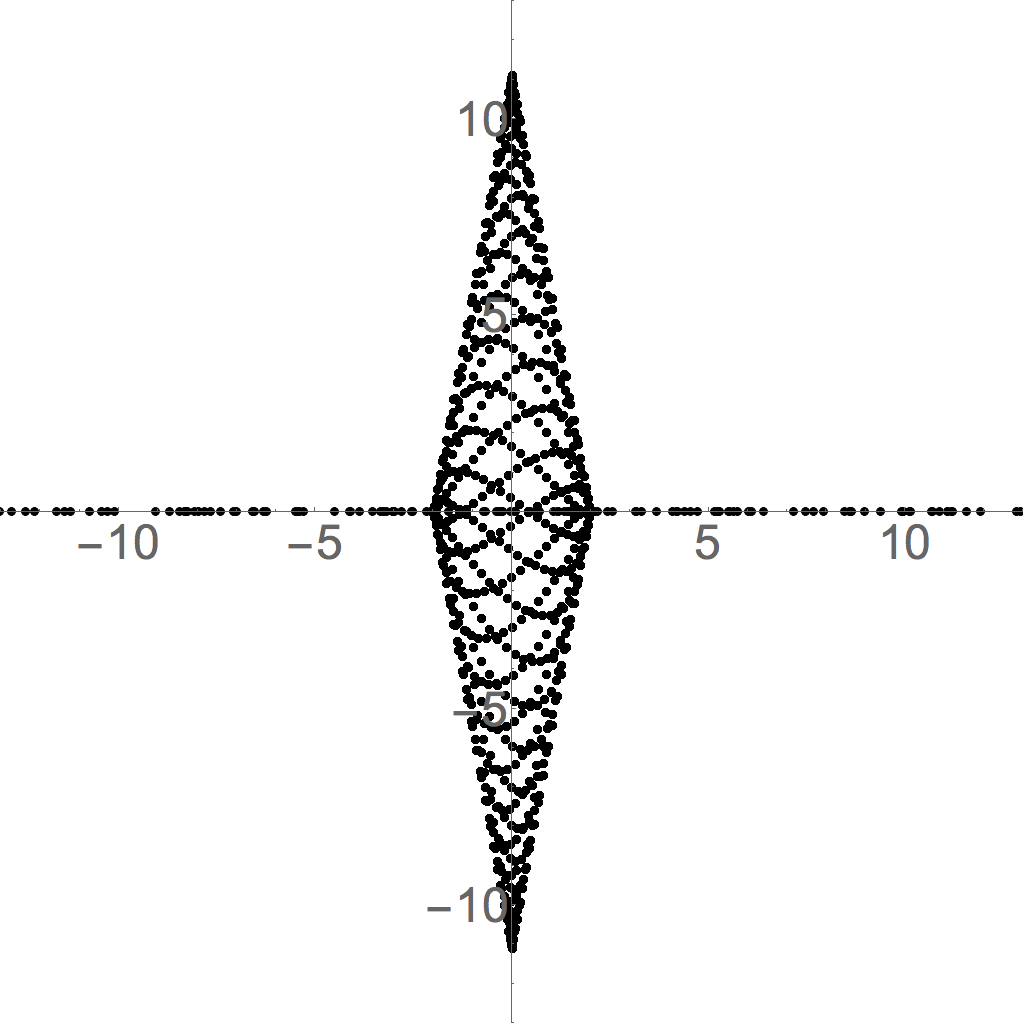}
\caption{\scriptsize $(31,1429,809)$}
\label{rhomfig3}
\end{subfigure}
\caption{For the given triples $(p,n,\omega)$, the nontrivial layers of the cyclic supercharacters $\sigma_\omega$ mod $pn$ are contained in rhombi. See Section \ref{rhomsec} for details.}
\end{figure}

\section{The present unknown}

While Theorem \ref{shapethm} provides concrete explanations of certain graphical behaviors, many remain elusive. It seems likely, however, that more could be handled in similar fashion to the ones above, armed with Theorem \ref{splitthm} and the language of Minkowski addition and multiplication. To close, we present Figure \ref{mysteryfig}, which provides a small gallery of plots yet unexplained. In Figure \ref{mysteryfig1}, the nontrivial layers appear to be contained in Minkowski sums of $3$ line segments. The nontrivial layers in Figure \ref{mysteryfig2} suggest Minkowski sums of ellipses, as in Section \ref{ellipsesec}. Patterns resembling the one in Figure \ref{mysteryfig3}, where $n=524287$ and $\omega=2$, seem to occur whenever $n$ has the form $2^j-1$ and $\omega=2$. We leave the reader with these observations.

\begin{figure}[ht]
\begin{subfigure}[b]{0.3\textwidth}
\includegraphics[width=\textwidth]{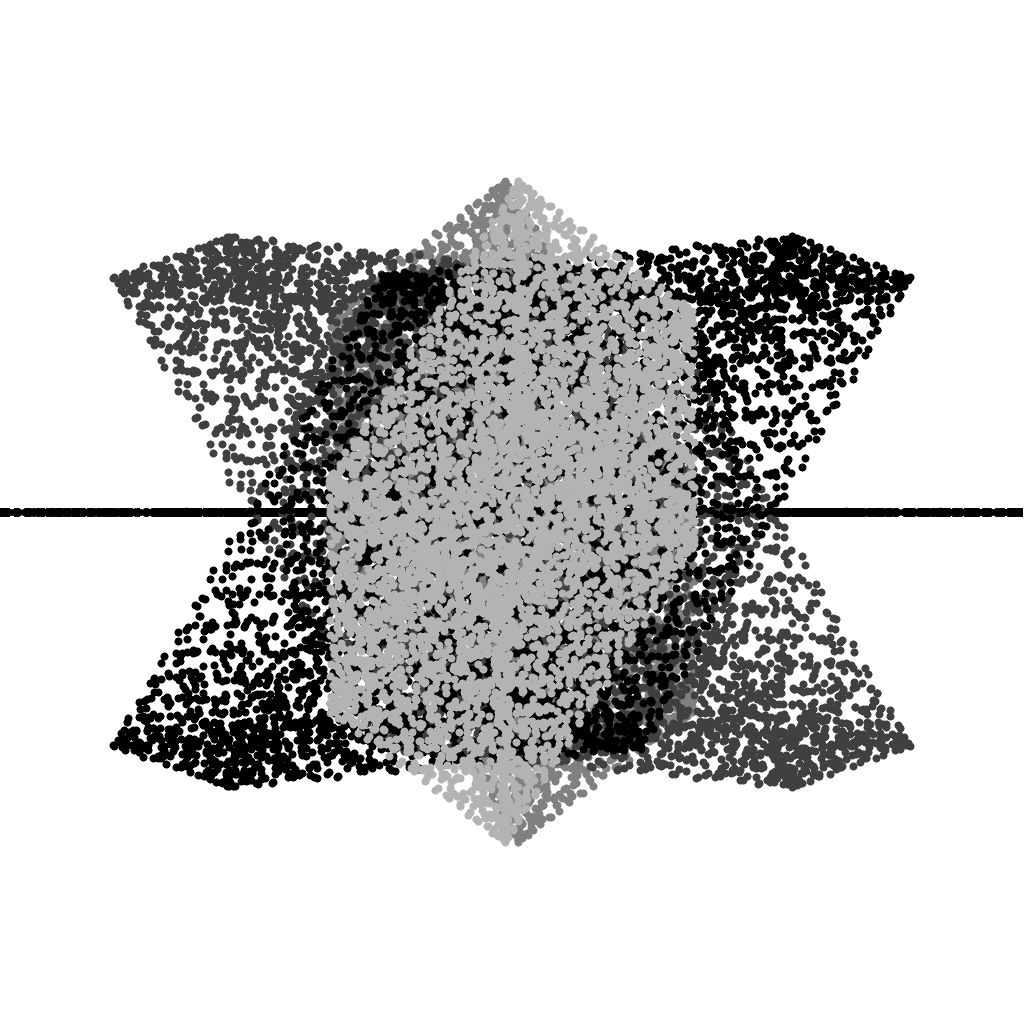}
\caption{\scriptsize $n=398157$, $\omega=1070$}
\label{mysteryfig1}
\end{subfigure}
\quad
\begin{subfigure}[b]{0.3\textwidth}
\includegraphics[width=\textwidth]{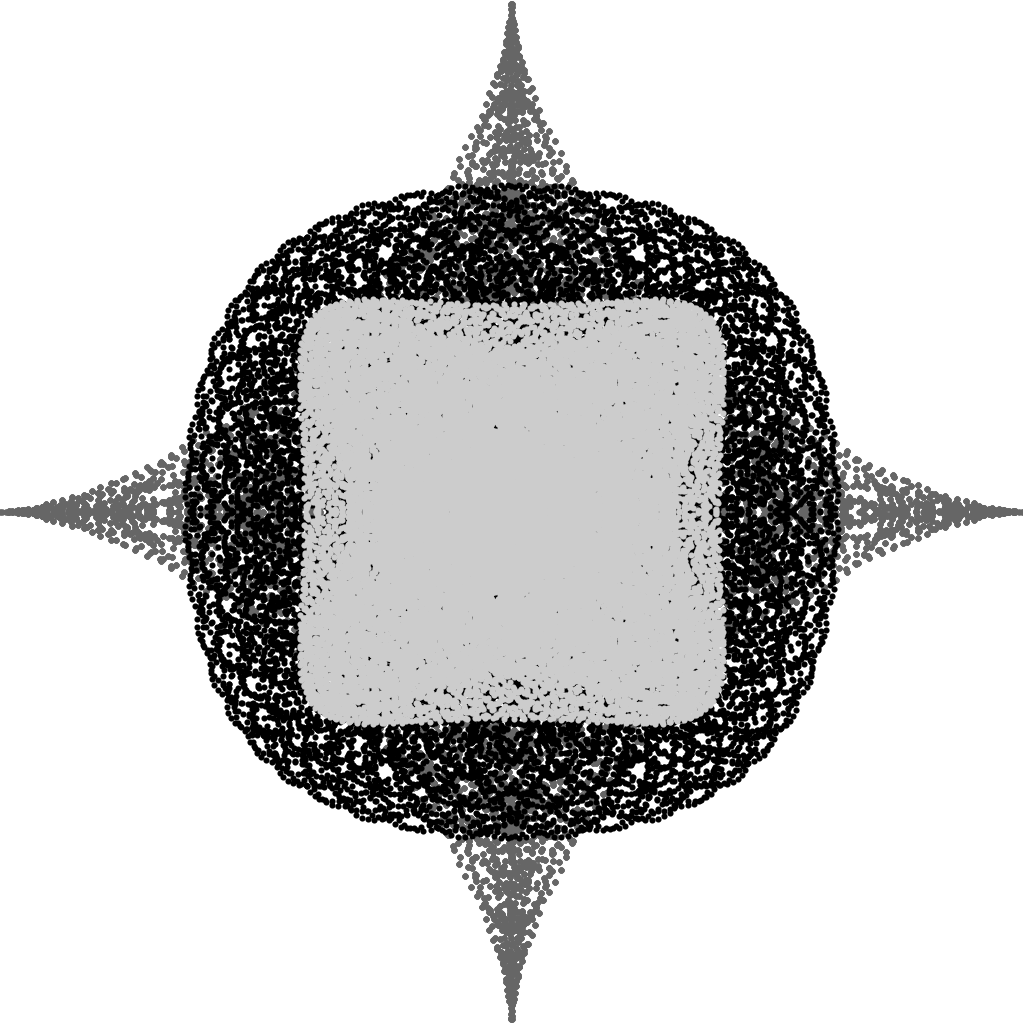}
\caption{\scriptsize $n=546975$, $\omega=593$}
\label{mysteryfig2}
\end{subfigure}
\quad
\begin{subfigure}[b]{0.3\textwidth}
\includegraphics[width=\textwidth]{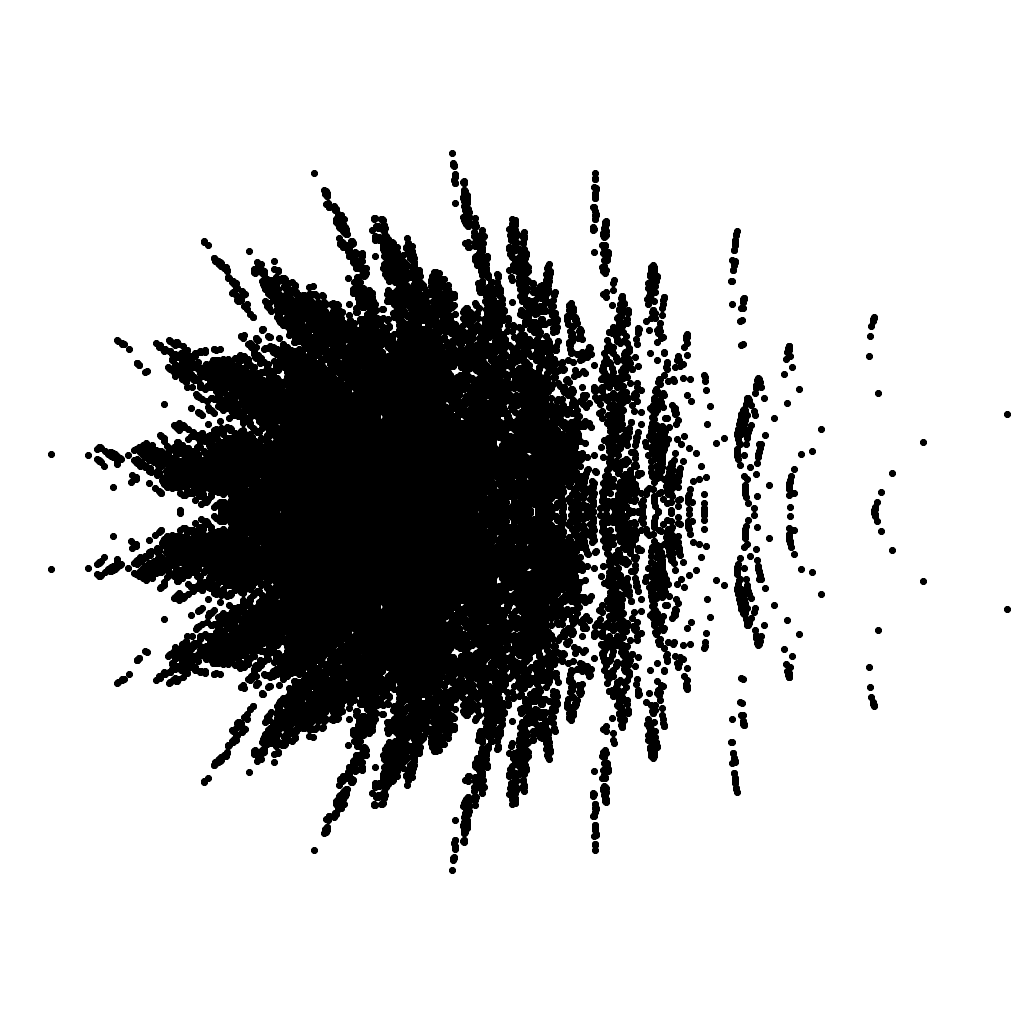}
\caption{\scriptsize $n=524287$, $\omega=2$}
\label{mysteryfig3}
\end{subfigure}
\caption{The plots of these cyclic supercharacters $\sigma_\omega$ mod $n$ have yet to be explained.}
\label{mysteryfig}
\end{figure}

\section*{Appendix}

We dedicate this section to proving Proposition \ref{dihprop}. For $A\subset \C$, define the \emph{backward cone} of $A$ to be the set $C(A)$ given by
\begin{equation*}
C(A)=\{ \lambda z \in \C : \lambda \in [0,1]\mbox{ and }z\in A\}.
\end{equation*}
If $A$ is compact, we define its \emph{outer boundary} $\partial(A)$ by
\begin{equation*}
\partial(A)=\{z\in A : A\cap \{\lambda z : \lambda > 1\}=\emptyset\}.
\end{equation*}
Notice that if $A$ is compact, then $\partial(A)=\partial(C(A))$, and that if $B$ is also compact, then $\partial(A\otimes B)\subset\partial(\partial(A)\otimes \partial(B))$, with equality if $A$ and $B$ are star shaped with center 0.

\begin{proof}[Proof of Proposition \ref{dihprop}]
Let $E_k$ (resp., $E_\prm$) be the edge of the polygon $\partial(P_k)$ (resp., $\partial(P_\prm)$) perpendicular to the real axis. By the preceding discussion, we see that
\begin{align}
\partial(P_k\otimes P_\prm)&=\partial(P_k)\otimes \partial(P_\prm)\nonumber\\
&= \partial(\{e(\tfrac{j}{k\prm}) : j=1,\ldots,k\prm\}\otimes (E_k\otimes E_\prm))\nonumber\\
&\subset \{e(\tfrac{j}{k\prm}) : j=1,\ldots,k\prm\}\otimes \partial(E_k\otimes E_\prm)\nonumber\\
&=\{e(\tfrac{j}{k\prm}) : j=1,\ldots,k\prm\}\otimes \partial(C(E_k\otimes E_\prm)).
\label{boundeq}
\end{align}
Let $a_k=-k\cos \frac{\pi}{k}$ (resp., $a_\prm=-\prm\cos\frac{\pi}{\prm}$), so that $a_k^{-1}E_k$ (resp., $a_\prm^{-1}E_\prm$) is the line segment connecting $1\pm i\tan\frac{\pi}{k}$ (resp., $1\pm i\tan\frac{\pi}{\prm}$). By \cite[Theorem 2.4(c)]{li2015}, $C(a_k^{-1}E_k\otimes a_\prm^{-1}E_\prm)$ is the set illustrated in Figure \ref{hexfig}, where 
\begin{align*}
z_1&=1+\tan^2 \tfrac{\pi}{\prm}\\
z_2&=1+\tan\tfrac{\pi}{k}\tan\tfrac{\pi}{\prm}+i(\tan\tfrac{\pi}{k}-\tan\tfrac{\pi}{\prm})\\
z_3&=1-\tan\tfrac{\pi}{k}\tan\tfrac{\pi}{\prm}+i(\tan\tfrac{\pi}{k}+\tan\tfrac{\pi}{\prm}).
\end{align*}
Since $C(E_k\otimes E_\prm)=a_ka_\prm C(a_k^{-1}E_k\otimes a_\prm^{-1}E_\prm)$, the result follows from combining standard trigonometric identities with \eqref{boundeq} and \eqref{diheq}.
\end{proof}

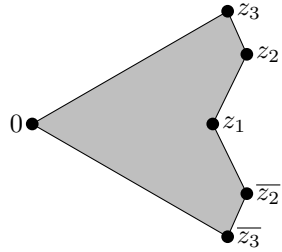
\begin{figure}[ht]
\begin{tikzpicture}[scale=3]
\coordinate [label=left:$0$] (A) at (0,0);
\coordinate (C) at (1,0);
\coordinate [label=right:$z_1$] (B) at (0.8,0);
\coordinate [label=right:$\overline{z_3}$] (E) at (0.866,-0.5);
\coordinate [label=right:$z_3$] (D) at (0.866,0.5);
\coordinate [label=right:$z_2$] (F) at (0.951,0.309);
\coordinate [label=right:$\overline{z_2}$] (G) at (0.951,-0.309);
\draw [fill=lightgray] (A) -- (D) -- (F) -- (B) -- (G) -- (E) -- (A);
\foreach \point in {A,B,D,E,F,G}
\fill [black] (\point) circle (0.8pt);
\end{tikzpicture}
\caption{The set $C(a_k^{-1}E_k\otimes a_\prm^{-1}E_\prm)$ defined in the proof of Proposition \ref{dihprop}.}
\label{hexfig}
\end{figure}

\section*{Acknowledgments}

The author thanks Trevor Hyde and the referee for their constructive suggestions and attention to detail.

\bibliographystyle{amsplain}
\bibliography{GCSFCM}
\end{document}